\documentclass{amsart}

\usepackage{amsmath,amssymb,amsthm}
\usepackage{cases}

\newtheorem{theorem}{Theorem}[section]
\newtheorem{proposition}[theorem]{Proposition}
\newtheorem{lemma}[theorem]{Lemma}
\newtheorem{remark}[theorem]{Remark}
\newtheorem{definition}[theorem]{Definition}

\begin{document}

\title{On a resolvent estimate for bidomain operators and its applications}

\author{Yoshikazu Giga}
\address{Graduate School of Mathematical Sciences, The University of Tokyo, 3-8-1 Komaba Meguro-ku Tokyo 153-8914, Japan. }
\curraddr{}
\email{labgiga@ms.u-tokyo.ac.jp}
\thanks{The first author is partly supported by the Japan Society for the Promotion of Science through the grant Kiban S (26220702) and Kiban B (16H03948)}

\author{Naoto Kajiwara}
\address{Graduate School of Mathematical Sciences, The University of Tokyo, 3-8-1 Komaba Meguro-ku Tokyo 153-8914, Japan.}
\curraddr{}
\email{kajiwara@ms.u-tokyo.ac.jp}
\thanks{}



\date{}

\dedicatory{}

\begin{abstract}
We study bidomain equations that are commonly used as a model to represent the electrophysiological wave propagation in the heart. 
We prove existence, uniqueness and regularity of a strong solution in $L^p$ spaces. 
For this purpose we derive an $L^\infty$ resolvent estimate for the bidomain operator by using a contradiction argument based on a blow-up argument. 
Interpolating with the standard $L^2$-theory, we conclude that bidomain operators generate $C_0$-analytic semigroups in $L^p$ spaces, which leads to construct a strong solution to a bidomain equation in $L^p$ spaces. 
\end{abstract}

\maketitle

\section{Introduction}

The bidomain model is a system related to intra- and extra-cellular electric potentials and some ionic variables. 
Mathematically, bidomain equations can be written as two partial differential equations coupled with a system of $m$ ordinary differential equations: 
\begin{align}
&\partial_{t}u+f(u,w)-\nabla\cdot(\sigma_{i}\nabla u_{i})=s_{i} & \mathrm{in}&~(0,\infty)\times\Omega,\label{inner} \\
&\partial_{t}u+f(u,w)+\nabla\cdot(\sigma_{e}\nabla u_{e})=-s_{e} & \mathrm{in}&~(0,\infty)\times\Omega,\label{exter} \\
&\partial_{t}w+g(u,w)=0 & \mathrm{in}&~(0,\infty)\times\Omega,\label{ODE} \\
&u=u_{i}-u_{e} & \mathrm{in}&~(0,\infty)\times\Omega,\label{action potential} \\
&\sigma_{i}\nabla u_{i}\cdot n=0,~\sigma_{e}\nabla u_{e}\cdot n=0 & \mathrm{on}&~(0,\infty)\times\partial\Omega,\label{boundary} \\
&u(0)=u_{0},~w(0)=w_{0} & \mathrm{in}&~\Omega.\label{initial data}
\end{align}

Here, functions $u_i$ and $u_e$ are intra- and extra-cellular electric potentials, $u$ is the transmembrane potential (or the action potential) and $w=w(t,x)\in\mathbb{R}^m(m\in\mathbb{N})$ is some ionic variables  (current, gating variables, concentrations, etc.). 
All these functions are unknown. 
On the other hand, the physical region occupied by the heart $\Omega\subset\mathbb{R}^d$, conductivity matrices $\sigma_{i,e}=\sigma_{i,e}(x)$, external applied current sources $s_{i,e}=s_{i,e}(t,x)$, total transmembrane ionic currents $f:\mathbb{R}\times\mathbb{R}^m\to\mathbb{R}$ and $g:\mathbb{R}\times\mathbb{R}^m\to\mathbb{R}^m$ and initial data $u_0$ and $w_0$ are given. 
The symbol $n$ denotes the unit outward normal vector to $\partial\Omega$. 
The reader is referred to the books \cite{CPS0} and \cite{KS} about mathematical physiology including  bidomain models. 

There are some literature about well-posedness of bidomain equations. 
First pioneering work is due to P. Colli-Franzone and G. Savar$\mathrm{\acute{e}}$ \cite{CS}. 
They introduced a variational formulation and derived existence, uniqueness and some regularity results in Hilbert spaces. 
Here, they assumed nonlinear terms $f,g$ are forms of $f(u,w)=k(u)+\alpha w$, $g(u,w)=-\beta u+\gamma w~(\alpha, \beta, \gamma\ge0)$ with a suitable growth condition on $k$. 
Examples include cubic-like FitzHugh-Nagumo model, which is the most fundamental electrophysiological model. 
However, other realistic models cannot be handled by their approach because nonlinear terms are limited. 
Later M. Veneroni \cite{V1} extended to their results by using fixed point argument and established well-posedness of more general and more realistic ionic models. 
These two papers discussed strong solutions by deriving further regularity of weak solutions. 
In 2009, Y. Bourgault, Y. Coudi$\mathrm{\acute{e}}$re and C. Pierre \cite{BCP} showed well-posedness of a strong solution in $L^2$ spaces. 
They transformed bidomain equations into an abstract evolution equation of the form 
\begin{align*}
\begin{cases}
\partial_t u + Au + f(u,w) =s, \\
\partial_t w + g(u,w) =0 
\end{cases}
\end{align*}
by introducing the bidomain operator $A$ in $L^2$ and modified source term $s$. 
Formally the bidomain operator is the harmonic mean of two elliptic operators, i.e. $(A_i^{-1}+A_e^{-1})^{-1}$ or $A_i(A_i+A_e)^{-1}A_e$, where $A_{i,e}$ is the elliptic operator $-\nabla\cdot(\sigma_{i,e}\nabla~\cdot~)$ with the homogeneous Neumann boundary condition. 
They proved that the bidomain operator is a non-negative self-adjoint operator by considering corresponding weak formulations. 
Since their framework is in $L^2$, well-posedness was only proved for $d\le 3$ in $L^2$ spaces. 

The main goal of this paper is to establish $L^p$-theory $(1<p<\infty)$ and $L^\infty$-theory for the bidomain operator with applications to bidomain equations. 
More explicitly, we shall prove that the bidomain operator forms an analytic semigroup $e^{-tA}$ both in $L^p$ and $L^\infty$. 
By this result we are able to construct a strong solution in $L^p$ for any space dimension $d$ (by taking $p$ large if necessary). 
Our result allows any locally Lipschitz nonlinear terms. 

To derive analyticity it is sufficient to derive resolvent estimates. 
For $L^p$ resolvent estimates a standard way is to use the Agmon's method (e.g. \cite{Lun}, \cite{Tan2}). 
The main idea of the method is as follows. 
If we have a $W^{2,p}(\Omega\times\mathbb{R})$ a priori estimate for the operator $A-e^{i\theta}\partial_{tt}$, then $A$ has an $L^p$ resolvent estimate. 
Unfortunately, it seems difficult to derive such a $W^{2,p}$ a priori estimate because of nonlocal structure of the bidomain operator. 
Thus we argue in a different way. 

We first establish an $L^\infty$ resolvent estimate for the bidomain operator by a contradiction argument including a blow-up argument. 
We then derive an $L^p$ resolvent estimate for $2\le p\le\infty$ by interpolating $L^2$ and $L^\infty$ results. 
The $L^p$-theory for $1<p<2$ is established by a duality argument. 
Note that a standard idea to derive an $L^\infty$ resolvent estimate due to Masuda-Stewart (see the third next paragraph) does not apply because their method is based on an $L^p$ resolvent estimate, which we would like to prove. 

A blow-up argument was first introduced by E. De Giorgi \cite{Dio} in order to study regularity of a minimal surface. 
It is also efficient to derive a priori estimates for solutions of a semilinear elliptic problem \cite{GS} and a semilinear parabolic problem \cite{G1}, \cite{GK}. 
Recently, K. Abe and the first author \cite{AG}, \cite{AG2} showed that the Stokes operator is a generator of an analytic semigroup on $C_{0,\sigma}(\Omega)$, the $L^\infty$-closure of $C_{c,\sigma}^\infty(\Omega)$ (the space of smooth solenoidal vector fields with compact support in $\Omega$) for some class of domains $\Omega$ including bounded and exterior domains by using a blow-up argument for a nonstationary problem. 
For a direct proof extending the Masuda-Stewart method for resolvent estimates, see \cite{AGH}. 
Suzuki \cite{Suz} showed analyticity of semigroups generated by higher order elliptic operators in $L^\infty$ spaces by a blow-up method even if the domain has only uniformly $C^1$ regularity for resolvent equations. 
Our approach is similar to his approach, but boundary conditions are different and our equations are systems. 
For the Dirichlet boundary condition, we can easily take a cut-off function and a test function. 
However, for the Neumann boundary condition, we have to take a cut-off function and a test function carefully so that we does not violate boundary conditions. 

Our method is based on a contradiction argument together with a blow-up argument. 
Let us explain a heuristic idea. 
Suppose that we would like to prove that 
\[|\lambda|\|u\|_{\infty}\le C\|s\|_{\infty}\]
with some $C>0$ independent of sufficiently large $\lambda$, $u$ and $s$ which satisfy the resolvent equation $\lambda u+Au=s$ in $\Omega$. 
Here, $\|\cdot\|_\infty$ denotes the $L^\infty(\Omega)$ norm. 
Suppose that the estimate were false. 
Then there would exists a sequence $\{\lambda_k\}_{k=1}^\infty$, $|\lambda_k|\to\infty$ and $\{u_k,s_k\}$ satisfy $\lambda_k u_k+Au_k=s_k$ in $\Omega$ such that $|\lambda_k|\|u_k\|_\infty > k\|s_k\|_\infty$. 
By normalizing $u_k$ to introduce $v_k=u_k/\|u_k\|_\infty$, we observe that $\|v_k\|_\infty=1$. 
We take $\{x_k\}_{k=1}^\infty \subset\Omega$ such that $|v_k(x_k)|>1/2$. 
We rescale $w_k(x)=v_k(x_k +x/|\lambda_k|^{1/2})$. 
This function solves the equation $e^{i\theta_k}w_k+A_kw_k=t_k$ in $\Omega_k$ with $A_k\to A_0$ if $A$ has a nice scaling property, where $e^{i\theta_k}=\lambda_k/|\lambda_k|$, $t_k(x)=s_k(x_k+x/|\lambda_k|^{1/2})/|\lambda_k|\|u_k\|_\infty$ and $\Omega_k:=|\lambda_k|^{1/2}(\Omega-x_k)$. 
Here, $A_0$ is the bidomain operator with a constant coefficient. 
Since $|\lambda_k|\to\infty$, the rescaled domain $\Omega_k$ converges to either the whole space or the half space. 
If $w_k$ converges to some $w$, then $w$ solves the limit equation $e^{i\theta_\infty}w+A_0w=0$ since $\|t_k\|_\infty <1/k$. 
If the convergence is strong enough, then the assumption $|w_k(0)|>1/2$ implies $|w(0)|\ge1/2$. 
However, if the solution of the limit equation $e^{i\theta_\infty}w+A_0w=0$ is unique, i.e. $w=0$, then we get a contradiction. 
The key step is a local `Compactness' of the blow-up sequence $\{w_k\}_{k=1}^\infty$ near zero to conclude $|w(0)|\ge1/2$ and `Uniqueness' of a blow-up limit. 

Let us explain some literatures for $L^\infty$-theory. 
For the Laplace operator or general elliptic operators it is well known that the corresponding semigroup is analytic in $L^\infty$-type spaces. 
K. Yosida \cite{Yos} considered the second order elliptic operator on $\mathbb{R}$. 
It was difficult to extend his method for multi-dimensional elliptic operators. 
K. Masuda \cite{Mas1}, \cite{Mas2} (see also \cite{Mas3}) first proved the analyticity of the semigroup generated by a general elliptic operator (even for higher-order elliptic operators) in $C_0(\mathbb{R}^d)$, the space of continuous functions vanishing at the space-infinity. 
For general domains, H. B. Stewart treated Dirichlet conditions \cite{Ste1} and general boundary conditions \cite{Ste2}. 
Their methods are based on a localization with $L^p$ results and interpolation inequalities. 
The reader may refer to the comprehensive book written by A. Lunardi \cite[Chapter 3]{Lun} for the Masuda-Stewart method which applies to many other cases. 
However, in our situations, we cannot apply these methods since we do not have $L^p$ estimates. 

Originally, bidomain equations were derived at a microscopic level. 
The cardiac cellular structure of the tissue can be viewed by disjoint unions of two regions separated by the interface, i.e. $\Omega=\Omega_i\cup\Omega_e\cup\Gamma$, where $\Omega_i$ and $\Omega_e$ are disjoint intra- and extra cellular domains and $\overline{\Gamma}=\partial\Omega_i\cap\partial\Omega_e$ is their interface called the active membrane. 
When we consider this model, the intra- and extra cellular potential $u_{i,e}$ are functions in $\overline{\Omega_{i,e}}$ respectively, and transmembrane potential $u=u_i-u_e$ is the function on $\overline{\Gamma}$. 
Bidomain equations are replaced to equations on $\Omega_i$, $\Omega_e$ and $\Gamma$ in this microscopic model. 
The dynamics inside the heart is much complicated. 
There are only a few papers (e.g. \cite{CS}, \cite{V2}) because of standard techniques and results on reaction diffusion equation systems cannot be directly applied. 
H. Matano and Y. Mori \cite{MM1} showed existence and uniqueness of a global classical solution for 3D cable model which is one of the microscopic cellular model by proving a uniform $L^\infty$ bound of solutions. 

Conversely at a macroscopic model, the cardiac tissue can be represented by a continuous model (called ``bi"domain model), i.e. $\Omega=\Omega_i=\Omega_e=\Gamma$ though each point of the heart $\Omega$ is one of the interior part $\Omega_i$ or exterior part $\Omega_e$ or their boundary $\Gamma$. 
Formal derivation from microscopic model to macroscopic model was shown by a homogenization process when a periodic cardiac structure \cite{KS}, \cite{NK}. 
The authors of \cite{CPS} showed a rigorous mathematical derivation of the macroscopic model by using the tools of the $\Gamma$-convergence theory. 
The paper \cite{ACS} studied the asymptotic behavior of the family of vectorial integral functionals, which is concerned with bidomain model, in the framework of $\Gamma$-convergence. 
The bidomain model is also used to analyze nonconvex mean curvature flow as a diffuse interface approximation \cite{A}, \cite{BPP}, \cite{BCP0}. 
Nonconvexity leads to the gradient flow of a nonconvex functional, which corresponds in general to an ill-posed parabolic problem. 
To study an ill-posed problem, it is often efficient to regularize it, for example by adding some higher order term, and then passing to the limit as the regularizing parameter goes to zero. 
However, papers \cite{A}, \cite{BPP}, \cite{BCP0} introduced completely different regularization, namely, to use bidomain equations, where hidden anisotrophy plays a key role. 
Recently in \cite{MM2}, interesting phenomena about stability of traveling wave solutions was found for bidomain Allen-Cahn equations, which is quite different from classical Allen-Cahn equations. 
This is also relevant to the hidden anisotropy of the bidomain model. 

The outline of this paper is as follows. 
In Section 2 after preparing a few notations, we state an $L^\infty$ resolvent estimate for bidomain equations, which is a key estimate of analyticity in $L^p$ and $L^\infty$ spaces. 
In Section 3 we give our proof of an $L^\infty$ resolvent estimate by using a blow-up argument. 
In Section 4 the system of bidomain equations is replaced by a single equation by using bidomain operators in $L^p$ spaces. 
Then we show existence and uniqueness of the solution. 
The method is based on a continuity method \cite{GT}. 
We also establish $L^p$ and $L^\infty$ resolvent estimates for bidomain operators based in our analysis in Section 3. 
In Section 5 to solve original problem ($\ref{inner}$)-($\ref{initial data}$) in $L^p$ we define bidomain operators and domains of their fractional powers in order to handle nonlinear terms $f,g$ having only locally Lipschitz continuity. 
From an $L^p$ resolvent estimate, we show bidomain operators are sectorial operators and then we derive existence, uniqueness and regularity of a strong solution to ($\ref{inner}$)-($\ref{initial data}$) in $L^p$ spaces. 

The authors are grateful to Professor Yoichiro Mori for guiding them to this problem. 

\section{Resolvent estimate for bidomain equations}

\subsection{Preliminaries, notations and definitions}

In this subsection we give a rigorous setting in order to state an $L^\infty$ resolvent estimate. 
We first recall the definition of uniformly $C^k$-domain for $k\ge1$ and function spaces $W_{loc}^{2,p}(\overline{\Omega})$. 

Let $B(x_0,r)$ be an open ball with center $x_0$, radius $r>0$, i.e. $B(x_0,r)=\{x\in\mathbb{R}^d\mid|x-x_0|<r\}$. 

\begin{definition}[Uniformly $C^k$-domain]
Let $\Omega\subset\mathbb{R}^d$ be a domain with $d\ge2$. 
We say that $\Omega$ is a uniformly $C^k$-domain $(k\ge1)$ if there exist $K>0$ and $r>0$ such that for each point $x_0\in\partial\Omega$ there exists a $C^k$ function $\gamma$ of $d-1$ variables $x'$ such that -upon relabeling, reorienting and rotation the coordinates axes if necessary- we have 
\begin{align*}
\Omega\cap B(x_0,r)&=\left\{x=(x',x_d)\in B(x_0,r)\mid x_d>\gamma(x')\right\}, \\
\|\gamma\|_{C^k(\mathbb{R}^{d-1})}&=\displaystyle\sup_{|\alpha|\le k,~x'\in\mathbb{R}^
{d-1}}|\partial_{x'}^\alpha \gamma(x')|\le K. 
\end{align*}
\end{definition}

\begin{definition}
We say $u\in W_{loc}^{2,p}(\overline{\Omega})$ if there exists $v\in W_{loc}^{2,p}(\mathbb{R}^d)$ such that $u=v$ a.e. in $\overline{\Omega}$. 
\end{definition}

The conductivity matrices $\sigma_{i,e}$ are functions of the space variable $x\in\overline{\Omega}$ with coefficients $C^1(\overline{\Omega})$ and satisfy the uniform ellipticity condition. 
Namely, we assume that there exist constants $0<\underline{\sigma}<\overline{\sigma}$ such that 
\begin{align}
\underline{\sigma}|\xi|^2\le\langle\sigma_{i,e}(x)\xi,\xi\rangle\le\overline{\sigma}|\xi|^2\label{UE}
\end{align}
for all $x\in\overline{\Omega}$ and $\xi\in\mathbb{R}^d$. 
Let $a=a(x)$ denote unit tangent vector at the point $x\in\partial\Omega$. 
Set the longitudinal conductances $k_{i,e}^l:\partial\Omega\to\mathbb{R}$ and the transverse conductances $k_{i,e}^t:\partial\Omega\to\mathbb{R}$ along the fibers. 
Commonly used conductance tensors are of the form (\cite{CGT}) 
\begin{align*}
\sigma_{i,e}(x)=k_{i,e}^t(x)I+(k_{i,e}^l(x)-k_{i,e}^t(x))a(x)\otimes a(x)~~(x\in\partial\Omega). 
\end{align*}
By this form we have the normal $n$ is the eigenvector of $\sigma_{i,e}$ whose eigenvalue is $k_{i,e}^t(x)$:  
\begin{align*}
\sigma_{i,e}(x)n(x)=k_{i,e}^t(x)n(x)~(x\in\partial\Omega). 
\end{align*}
When the model is constructed, these are naturally considered. 
Under these assumptions of $\sigma_{i,e}$,  we have the property of boundary conditions: 
\begin{align}
\sigma_{i,e}\nabla u\cdot n=0\Leftrightarrow\nabla u\cdot n=0~~~{\rm on}~\partial\Omega. \label{EV}
\end{align}

Source terms $s_{i,e}$ also have important property. 
In physiology no current flow outside through boundary $\partial\Omega$ and the intra- and extra-cellular media communicate electrically through the transmembrane. 
Hereafter we assume current conservation; 
\begin{align*}
\int_\Omega(s_i(t)+s_e(t))dx=0~(t\ge0). 
\end{align*}
This is nothing but the compatibility condition for bidomain equations. 
This averaging zero condition is used when we transform the system of bidomain equations (\ref{inner})-(\ref{initial data}) into single equation (\ref{z})-(\ref{u_e}). 

\subsection{Resolvent estimate}

We consider the following resolvent equations 
\begin{align*}
(\ast)
\begin{cases}
\lambda u-\nabla\cdot(\sigma_i\nabla u_i)=s\hspace{5mm}&{\rm in} \ \Omega,\\
\lambda u+\nabla\cdot(\sigma_e\nabla u_e)=s\hspace{5mm}&{\rm in} \ \Omega, \\
u=u_i-u_e\hspace{5mm}&{\rm in} \ \Omega, \\
\sigma_i\nabla u_i\cdot n=0,~\sigma_e\nabla u_e\cdot n=0\hspace{5mm}&{\rm on} \ \partial\Omega, 
\end{cases}
\end{align*}
corresponding to (\ref{inner})-(\ref{initial data}). 
These equations come from the Laplace transformation of linear part of bidomain equations. 

Let us state an $L^\infty$ resolvent estimate. 
We set $\Sigma_{\theta,M}:=\{\lambda\in\mathbb{C}\setminus\{0\}\mid |\arg\lambda|<\theta, M< |\lambda|\}$ and $N(u,u_i,u_e,\lambda)$ of the form  
\begin{align*}
N(u,u_i,u_e,\lambda):=\sup_{x\in\Omega}\left(|\lambda||u(x)|+|\lambda|^{1/2}\left(|\nabla u(x)|+|\nabla u_i(x)|+|\nabla u_e(x)|\right)\right). 
\end{align*}

\begin{theorem}[$L^\infty$ resolvent estimate for bidomain equations]\label{main result}
Let $\Omega\subset\mathbb{R}^d$ be a uniformly $C^2$-domain and $\sigma_{i,e}\in C^1(\overline{\Omega},\mathbb{S}^d)$ satisfy {\rm (\ref{UE})} and {\rm (\ref{EV})}. 
Then for each $\varepsilon\in(0,\pi/2)$ there exist $C>0$ and $M>0$ such that 
\begin{align*}
N(u,u_i,u_e,\lambda)\le C\|s\|_{L^{\infty}(\Omega)}
\end{align*}
for all $\lambda\in\Sigma_{\pi-\varepsilon,M}$, $s\in L^{\infty}(\Omega)$ and strong solutions $u,u_{i,e}\in \bigcap_{n<p<\infty}W^{2,p}_{loc}(\overline{\Omega})\cap W^{1,\infty}(\Omega)$ of $(\ast)$. 
\end{theorem}

\begin{remark}
${\rm(i)}$ It is impossible to derive an estimate $|\lambda|\|u_{i,e}\|_\infty\le C\|s\|_{L^\infty(\Omega)}$ because if $(u,u_i,u_e)$ is a triplet of strong solutions then so is $(u,u_i+c,u_e+c)$ for all $c\in\mathbb{R}$. \\
$\rm (ii)$ By the Sobolev embedding theorem {\rm \cite{AF}}, 
\begin{align*}
\bigcap_{n<p<\infty}W^{2,p}_{loc}(\overline{\Omega})\cap W^{1,\infty}(\Omega)\subset\bigcap_{0<\alpha<1}C^{1+\alpha}(\overline{\Omega}). 
\end{align*}
Hence $(u,u_i,u_e)$ are $C^1$ functions and the left-hand side of the resolvent estimate makes sense. 
\end{remark}

\section{Proof of an $L^\infty$ resolvent estimate}
\begin{proof}[Proof of Theorem \ref{main result}]
We divide the proof into five steps. 
The first two steps are reformulation of equations and estimates.
The last three steps (compactness, characterization of the limit and uniqueness) are crucial. 

\begin{flushleft}
\textbf {Step 1 (Normalization)}
\end{flushleft}

We argue by contradiction. 
Suppose that the statement were false. 
Then there would exist $\varepsilon\in(0,\pi/2)$, for any $k\in\mathbb{N}$ there would exist $\lambda_k=|\lambda_k|e^{i\theta_k}\in\Sigma_{\pi-\varepsilon,k}$, $s_k\in L^\infty(\Omega)$ and $u_k, u_{ik}, u_{ek}\in\bigcap_{n<p<\infty}W^{2,p}_{loc}(\overline{\Omega})\cap W^{1,\infty}(\Omega)$ which are strong solutions of resolvent equations
\begin{equation*}
\begin{cases}
\lambda_k u_k-\nabla\cdot(\sigma_{i}\nabla u_{ik})=s_k\hspace{5mm}&{\rm in} \ \Omega, \\
\lambda_k u_k+\nabla\cdot(\sigma_{e}\nabla u_{ek})=s_k\hspace{5mm}&{\rm in} \ \Omega, \\
u_k=u_{ik}-u_{ek}\hspace{5mm}&{\rm in} \ \Omega, \\
\sigma_{i}\nabla u_{ik}\cdot n=0,~\sigma_{e}\nabla u_{ek}\cdot n=0\hspace{5mm}&{\rm on} \ \partial\Omega, 
\end{cases}
\end{equation*}
with an $L^\infty$ estimate $N(u_k,u_{ik},u_{ek},\lambda_k)> k\|s_k\|_{L^{\infty}(\Omega)}$. 

We set 
\begin{align*}
\left(
\begin{array}{c}
v_k\\
v_{ik}\\
v_{ek}\\
\tilde{s}_k
\end{array}
\right)
&:=
\frac{1}{N(u_k,u_{ik},u_{ek},\lambda_k)}
\left(
\begin{array}{c}
|\lambda_k| u_k\ \\
|\lambda_k| u_{ik}\\
|\lambda_k| u_{ek}\\
s_k
\end{array}
\right).
\end{align*}
Then we get normalized resolvent equations of the form
\begin{equation*}
\begin{cases}
e^{i\theta_k}v_k-\frac{1}{|\lambda_k|}\nabla\cdot(\sigma_{i}\nabla v_{ik})=\tilde{s}_k\hspace{5mm}&{\rm in} \ \Omega, \\
e^{i\theta_k}v_k+\frac{1}{|\lambda_k|}\nabla\cdot(\sigma_{e}\nabla v_{ek})=\tilde{s}_k\hspace{5mm}&{\rm in} \ \Omega, \\
v_k=v_{ik}-v_{ek}\hspace{5mm}&{\rm in} \ \Omega, \\
\sigma_{i}\nabla v_{ik}\cdot n=0,~\sigma_{e}\nabla v_{ek}\cdot n=0\hspace{5mm}&{\rm on} \ \partial\Omega, 
\end{cases}
\end{equation*}
with estimates $\frac{1}{k}>\|\tilde{s}_k\|_{L^\infty(\Omega)}$ and 
\begin{align*}
&N\left(\frac{v_k}{|\lambda_k|},\frac{v_{ik}}{|\lambda_k|},\frac{v_{ek}}{|\lambda_k|},\lambda_k\right)\\
=&\sup_{x\in\Omega}\left(|v_k(x)|+|\lambda_k|^{-1/2}\left(|\nabla v_k(x)|+|\nabla v_{ik}(x)|+|\nabla v_{ek}(x)|\right)\right)\\
=&1. 
\end{align*}

\begin{flushleft}
\textbf{Step 2 (Rescaling)}
\end{flushleft}

Secondly, we rescale variables near maximum points of normalized $N$. 
By definition of supremum there exists $\{x_k\}_{k=1}^{\infty}\subset\Omega$ such that 
\begin{align*}
|v_k(x_k)|+|\lambda_k|^{-1/2}\left(|\nabla v_k(x_k)|+|\nabla v_{ik}(x_k)|+|\nabla v_{ek}(x_k)|\right)>\frac{1}{2}
\end{align*}
for all $k\in\mathbb{N}$. 
We rescale functions $\{(w_k,w_{ik},w_{ek})\}_{k=1}^\infty$, $\{t_k\}_{k=1}^\infty$, matrices $\{(\sigma_{ik},\sigma_{ek})\}_{k=1}^\infty$ and domain $\Omega_k$ with respect to $x_k$. 
Namely, we set 
\begin{align*}
\left(
\begin{array}{c}
w_k\\
w_{ik}\\
w_{ek}
\end{array}
\right)(x)
:=&
\left(
\begin{array}{c}
v_k\\
v_{ik}\\
v_{ek}
\end{array}
\right)\left(x_k+\frac{x}{|\lambda_k|^{1/2}}\right), \\
t_k(x):=&\tilde{s}_k\left(x_k+\frac{x}{|\lambda_k|^{1/2}}\right), \\
\sigma_{ik}(x):=\sigma_i\left(x_k+\frac{x}{|\lambda_k|^{1/2}}\right),~~&
\sigma_{ek}(x):=\sigma_e\left(x_k+\frac{x}{|\lambda_k|^{1/2}}\right), \\
\Omega_k:=&|\lambda_k|^{1/2}(\Omega-x_k). 
\end{align*}
By changing variables $\Omega\ni x\mapsto|\lambda_k|^{1/2}(x-x_k)\in\Omega_k$, we notice that our equations and our estimates can be rewritten of the form 
\begin{equation*}
\begin{cases}
e^{i\theta_k}w_k-\nabla\cdot(\sigma_{ik}\nabla w_{ik})=t_k\hspace{5mm}&{\rm in} \ \Omega_k, \\
e^{i\theta_k}w_k+\nabla\cdot(\sigma_{ek}\nabla w_{ek})=t_k\hspace{5mm}&{\rm in} \ \Omega_k, \\
w_k=w_{ik}-w_{ek}\hspace{5mm}&{\rm in} \ \Omega_k, \\
\sigma_{ik}\nabla w_{ik}\cdot n_k=0,~\sigma_{ek}\nabla w_{ek}\cdot n_k=0\hspace{5mm}&{\rm on} \ \partial\Omega_k, 
\end{cases}
\end{equation*}
with estimates 
\begin{align*}
&\frac{1}{k}>\|t_k\|_{L^\infty(\Omega_k)}, \\
&|w_k(0)|+|\nabla w_k(0)|+|\nabla w_{ik}(0)|+|\nabla w_{ek}(0)|>\frac{1}{2}, \\
&\displaystyle\sup_{x\in\Omega_k}\left(|w_k(x)|+|\nabla w_k(x)|+|\nabla w_{ik}(x)|+|\nabla w_{ek}(x)|\right)=1, 
\end{align*}
where $n_k$ denotes the unit outer normal vector to $\Omega_k$. 
Here, we remark that unknown functions $w_{ik}$ and $w_{ek}$ are defined up to an additive constant. 
So without loss of generality we may assume that $w_{ik}(0):=0$. 

\begin{flushleft}
\textbf{Step 3 (Compactness)}
\end{flushleft}

In this step, we will show local uniform boundedness for $\{(w_k, w_{ik}, w_{ek})\}_{k=1}^\infty$. 
If these sequences are bounded, one can take subsequences $\{(w_{k_l}, w_{i{k_l}}, w_{e{k_l}})\}_{l=1}^\infty$ which uniformly convergences in the norm $C^1$ on each compact set.  
We need to divide two cases. 
One is the case $\Omega_\infty=\mathbb{R}^d$ and the other is the case $\Omega_\infty=\mathbb{R}_+^d$ up to translation and rotation, where $\Omega_\infty$ is the limit of $\Omega_k$. \\
We set $d_k= {\rm dist}(0,\partial\Omega_k)=|\lambda_k|^{1/2}{\rm dist}(x_k,\partial\Omega)$ and $D:=\displaystyle{\liminf_{k\rightarrow\infty}}\ d_{k}$. 

\begin{flushleft}
\textbf{Case(3-i) $D=\infty$}
\end{flushleft}

In this case $\Omega_\infty=\mathbb{R}^d$ (See \cite{Suz}). 
Let cut-off function $\rho\in C_0^\infty(\mathbb{R}^d)$ be such that $\rho(x)\equiv 1$ for $|x|\le 1$ and $\rho(x)\equiv 0$ for $|x|\ge 3/2$. 
We localize functions $w_k$, $w_{ik}$, $w_{ek}$ as follows 
\begin{align*}
\left(
\begin{array}{c}
w_k^\rho\\
w_{ik}^\rho\\
w_{ek}^\rho
\end{array}
\right)
:=\rho
\left(
\begin{array}{c}
w_k\\
w_{ik}\\
w_{ek}
\end{array}
\right)\hspace{5mm}{\rm in}\ \ \Omega_k. 
\end{align*}
By multiplying rescaled resolvent equations by $\rho$, we consider the following localized equations 
\begin{align}
&e^{i\theta_k}w_k^\rho-\nabla\cdot(\sigma_{ik}\nabla w_{ik}^\rho)=t_k\rho+I_{ik}&{\rm in}& \ \Omega_k,\label{localin} \\
&e^{i\theta_k}w_k^\rho+\nabla\cdot(\sigma_{ek}\nabla w_{ek}^\rho)=t_k\rho+I_{ek}&{\rm in}& \ \Omega_k,\label{localex} \\
&w_k^\rho=w_{ik}^\rho-w_{ek}^\rho&{\rm in}& \ \Omega_k,\label{in-ex} \\
&\sigma_{ik}\nabla w_{ik}^\rho\cdot n_k=0,~\sigma_{ek}\nabla w_{ek}^\rho\cdot n_k=0&{\rm on}& \ \partial\Omega_k, \label{BD} 
\end{align}
where 
\begin{align*}
I_{ik}=
-\sum_{1\le m,n\le d}\Big\{\bigl((\sigma_{ik})_{mn}\bigr)_{x_m}\rho_{x_n}w_{ik}+(\sigma_{ik})_{mn}\rho_{x_mx_n}w_{ik}\\
+(\sigma_{ik})_{mn}\rho_{x_n}(w_{ik})_{x_m}+(\sigma_{ik})_{mn}\rho_{x_m}(w_{ik})_{x_n}\Big\}, 
\end{align*}
\begin{align*}
I_{ek}=
\ \sum_{1\le m,n\le d}\Big\{\bigl((\sigma_{ek})_{mn}\bigr)_{x_m}\rho_{x_n}w_{ek}+(\sigma_{ek})_{mn}\rho_{x_mx_n}w_{ek}\\
+(\sigma_{ek})_{mn}\rho_{x_n}(w_{ek})_{x_m}+(\sigma_{ek})_{mn}\rho_{x_m}(w_{ek})_{x_n}\Big\}
\end{align*}
are lower order terms of $w_{ik}$ and $w_{ek}$. 
Here, we take sufficiently large $k$ such that $B(0,2)\subset\Omega_k$. 

Take some $p>n$ and apply $W^{2,p}(\Omega_k)$ a priori estimate for second order elliptic operators $-\nabla\cdot(\sigma_{ik}\nabla\cdot)$, which have the oblique boundary (\ref{BD}). 
By (\ref{localin}) there exists $C>0$ independent of $k\in\mathbb{N}$ such that 
\begin{align*}
&\|w_{ik}^\rho\|_{W^{2,p}(\Omega_k)}\\
\le&C\left(\|w_{ik}^\rho\|_{L^p(\Omega_k)}+\|w_k^\rho\|_{L^p(\Omega_k)}+\|t_k\rho\|_{L^p(\Omega_k)}+\|I_{ik}\|_{L^p(\Omega_k)}\right) \\
\le&C|B(0,2)|^{1/p}\left(\|w_{ik}^\rho\|_{L^\infty(\Omega_k)}+\|w_k^\rho\|_{L^\infty(\Omega_k)}+\|t_k\rho\|_{L^\infty(\Omega_k)}+\|I_{ik}\|_{L^\infty(\Omega_k)}\right)\\
=:&C|B(0,2)|^{1/p}\left(I+I\hspace{-.1em}I+I\hspace{-.1em}I\hspace{-.1em}I+I\hspace{-.1em}V\right), 
\end{align*}
where we use H${\rm \ddot{o}}$lder inequality in the second inequality. 
The first term $I$ is uniformly bounded in $k$ since $w_{ik}(0)=0$ and $\|\nabla w_{ik}\|_{L^\infty(\Omega_k)}\le1$. 
The second term $I\hspace{-.1em}I$ and the third term $I\hspace{-.1em}I\hspace{-.1em}I$ are also uniformly bounded in $k$ since $\|w_k\|_{L^\infty(\Omega_k)}\le1$, $\|\rho\|_{L^\infty(\Omega_k)}\le1$ and $\|t_k\|_{L^\infty(\Omega_k)}<1/k$. 
Finally the forth term $I\hspace{-.1em}V$ is also uniformly bounded in $k$ since 
\begin{align*}
I\hspace{-.1em}V&\le C(d,\sup_k\|\sigma_{ik}\|_{W^{1,\infty}(\Omega_k)})\|w_{ik}\|_{W^{1,\infty}(\Omega_k)}\\
&\le C. 
\end{align*}
Here, the constant $C$ may differ from line to line. 
Therefore the sequence $\{w_{ik}^\rho\}_{k=1}^\infty$ is uniformly bounded in $W^{2,p}(\Omega_k)$. 
Functions $\{w_{ek}^\rho\}_{k=1}^\infty$ and $\{w_{k}^\rho\}_{k=1}^\infty$ are also uniformly bounded in $W^{2,p}(\Omega_k)$ since the same calculation as above and (\ref{in-ex}). 
Here, $\Omega_k$ depends on $k\in\mathbb{N}$. 
By zero extension from $\Omega_k$ to $\mathbb{R}^d$, we have $\{(w_k^\rho, w_{ik}^\rho, w_{ek}^\rho)\}_{k=1}^\infty$ is uniform bounded in the norm $\left(W^{2,p}(\mathbb{R} ^d)\right)^3$. 
Thus we are able to take subsequences $\{(w_{k_l}^\rho, w_{i k_l}^\rho, w_{e k_l}^\rho)\}_{l=1}^\infty$ and $w, w_i, w_e\in W^{2,p}(\mathbb{R}^d)$ such that 
\begin{align*}
\left(
\begin{array}{c}
w_{k_l}^\rho\\
w_{i{k_l}}^\rho\\
w_{e{k_l}}^\rho
\end{array}
\right)
\to
\left(
\begin{array}{c}
w\\
w_i\\
w_e
\end{array}
\right)~{\rm in~the~norm~}C^1(\mathbb{R}^d)~{\rm as}~l\to\infty, 
\end{align*}
by Rellich's compactness theorem \cite{AF}. 
Since 
\begin{align*}
|w_{k_l}(0)|+|\nabla w_{k_l}(0)|+|\nabla w_{i{k_l}}(0)|+|\nabla w_{e{k_l}}(0)|>\frac{1}{2}, 
\end{align*}
we get 
\begin{align*}
|w(0)|+|\nabla w(0)|+|\nabla w_i(0)|+|\nabla w_e(0)|\ge\frac{1}{2}. 
\end{align*}

\begin{flushleft}
\textbf{Case(3-ii) $D<\infty$}
\end{flushleft}

In this case $\Omega_\infty=\mathbb{R}_+^d$ up to translation and rotation (See \cite{Suz}). 
Let cut-off functions $\{\rho_k\}_{k=1}^\infty\subset C_0^\infty(B(0,2))$ satisfy $0\le\rho_k\le1$, $\partial\rho_k/\partial n_k=0$ on $\partial\Omega_k$, $\rho_k\equiv1$ on $B(0,1)$ and there exists $K>0$ such that $\sup_{k\in\mathbb{N}}\|\rho_k\|_{W^{2,\infty}(B(0,2))}\le K$. 
This is of course possible \cite[Appendix B]{AGSS}. 
Take a bounded $C^2$-domain $\Omega_k'$ with the unit outer normal vector $\tilde{n}_k$ such that $\Omega_k'\subset\Omega_k\cap B(0,2)$, $B(0,1)\cap\partial\Omega_k\subset\partial\Omega_k'$ and $\partial\rho_k/\partial\tilde{n}_k=0$ on $\partial\Omega_k'$. 

We argue in the same way as Case(3-i), we localize $(w_k,w_{ik},w_{ek})$ by multiplying $\rho_k$ and get $W^{2,p}(\Omega_k')$ a priori estimate for $(w_k^{\rho_k},w_{ik}^{\rho_k},w_{ek}^{\rho_k})$. 
We consider some neighborhood $U\subset\Omega_k'$ near the origin such that $\partial(U\cap\overline{\mathbb{R}_+^d})$ is smooth. 
Then we can take subsequences $\{(w_{k_l}^{\rho_{k_l}}, w_{i{k_l}}^{\rho_{k_l}}, w_{e{k_l}}^{\rho_{k_l}})\}_{l=1}^\infty$ and $w, w_i, w_e\in W^{2,p}(U\cap\overline{\mathbb{R}_+^d})$ such that 
\begin{align*}
\left(
\begin{array}{c}
w_{k_l}^{\rho_{k_l}}\\
w_{i{k_l}}^{\rho_{k_l}}\\
w_{e{k_l}}^{\rho_{k_l}}
\end{array}
\right)
\to
\left(
\begin{array}{c}
w\\
w_i\\
w_e
\end{array}
\right)~{\rm in~the~norm~}C^1(U\cap\overline{\mathbb{R}^d_+})~{\rm as}~l\to\infty. 
\end{align*}
As in Case (3-i), we get the same inequality. 

In this step, we are able to conclude that $w\not\equiv0$ and $w_{i,e}$ are not constants on some neighborhood near the origin. 

\begin{flushleft}
\textbf{Step 4 (Characterization of the limit)}
\end{flushleft}

Let us explain resolvent equations of $w_{k_l}$, $w_{i{k_l}}$, $w_{e{k_l}}$ tend to the limit equation 
\begin{equation}\label{infty}
\begin{cases}
e^{i\theta_\infty}w-\nabla\cdot(\sigma_{i\infty}\nabla w_i)=0\hspace{5mm}&{\rm in} \ \Omega_\infty, \\
e^{i\theta_\infty}w+\nabla\cdot(\sigma_{e\infty}\nabla w_e)=0\hspace{5mm}&{\rm in} \ \Omega_\infty, \\
w=w_i-w_e\hspace{5mm}&{\rm in} \ \Omega_\infty, \\
\sigma_{i\infty}\nabla w_i\cdot n_\infty=0,~\sigma_{e\infty}\nabla w_e\cdot n_\infty=0\hspace{5mm}&{\rm on} \ \partial\Omega_\infty, 
\end{cases}
\end{equation}
in the weak sense, where $\theta_\infty=\lim_{k\to\infty}\theta_k$, $\sigma_{i\infty},\sigma_{e\infty}$ are constant coefficients matrices defined as below which satisfy uniform ellipticity condition and $n_\infty$ is unit outer normal vector $(0,\cdots,0,-1)$ when $\Omega_\infty=\mathbb{R}_+^d$. 
If $\Omega_\infty=\mathbb{R}^d$, we do not need to consider boundary conditions. 

We have $w, w_i, w_e\in\bigcap_{n<p<\infty}W_{loc}^{2,p}(\overline{\Omega_\infty})\cap W^{1,\infty}(\Omega_\infty)$ and 
\begin{align*}
\left(
\begin{array}{c}
w_{k_l}\\
\nabla w_{i{k_l}}\\
\nabla w_{e{k_l}}
\end{array}
\right)
\to
\left(
\begin{array}{c}
w\\
\nabla w_i\\
\nabla w_e
\end{array}
\right)~{\rm weak\ast~in~}L^\infty(\Omega_\infty)~{\rm as}~l\to\infty
\end{align*}
since $\displaystyle\sup_{x\in\Omega_k}\left(|w_k(x)|+|\nabla w_k(x)|+|\nabla w_{ik}(x)|+|\nabla w_{ek}(x)|\right)=1$. 

\begin{flushleft}
\textbf{Case(4-i)} $\Omega_\infty=\mathbb{R}^d$
\end{flushleft}

\begin{proposition}\label{limit1}
The limit $w, w_i, w_e\in\bigcap_{n<p<\infty}W_{loc}^{2,p}(\mathbb{R}^d)\cap W^{1,\infty}(\mathbb{R}^d)$ satisfy that for any $\phi_{i,e}\in C_0^\infty(\mathbb{R}^d)$ 
\begin{align*}
\begin{cases}
e^{i\theta\infty}(w,\phi_i)_{L^2(\mathbb{R}^d)}+(\sigma_{i\infty}\nabla w_i,\nabla\phi_i)_{L^2(\mathbb{R}^d)}=0, \\
e^{i\theta\infty}(w,\phi_e)_{L^2(\mathbb{R}^d)}-(\sigma_{e\infty}\nabla w_e,\nabla\phi_e)_{L^2(\mathbb{R}^d)}=0, \\
w=w_i-w_e, 
\end{cases}
\end{align*}
where $\theta_\infty=\lim_{k\to\infty}\theta_k$ and $\sigma_{i\infty},\sigma_{e\infty}$ are constant coefficients matrices which satisfy uniform ellipticity condition. 
Here, $(\cdot,\cdot)_{L^2(\mathbb{R}^d)}$ denotes $L^2$-inner product. 
\end{proposition}

\begin{proof}[Proof of Proposition \ref{limit1}]
For each function $\eta\in C_0^\infty(\mathbb{R}^d)$, there exists $k_\eta\in\mathbb{N}$ such that ${\rm supp}\,\eta\subset\Omega_k$ for $k_\eta\le k$. 
Since ${\rm supp}\,\eta$ is compact, there exist $w_{k_l}$, $w_{i{k_l}}$, $w_{e{k_l}}$ such that 
\begin{align*}
\left(
\begin{array}{c}
w_{k_l}\\
w_{i{k_l}}\\
w_{e{k_l}}
\end{array}
\right)
\to
\left(
\begin{array}{c}
w\\
w_i\\
w_e
\end{array}
\right)~{\rm weakly~on~} W^{2,p}({\rm supp}\,\eta)~{\rm as~}l\to\infty
\end{align*}
for all $n<p<\infty$. 
Now we have to determine $\sigma_{i\infty}$ and $\sigma_{e\infty}$. 
For matrix $A=\{a_{mn}\}_{1\le m,n\le d}$, set $\|A\|:=\max_{1\le m,n\le d}|a_{mn}|$. 
Since $\sigma_i$ is uniformly continuous, for each $\varepsilon>0$ there exists $\delta>0$ such that if $\left|x_k+\frac{x}{|\lambda_k|^{1/2}}-x_k\right|=\left|\frac{x}{|\lambda_k|^{1/2}}\right|<\delta$ then $\left\|\sigma_i\left(x_k+\frac{x}{|\lambda_k|^{1/2}}\right)-\sigma_i(x_k)\right\|=\|\sigma_{ik}(x)-\sigma_{ik}(0)\|<\varepsilon$. 
We can take $k_0\in\mathbb{N}$ such that $\left|\frac{x}{|\lambda_k|^{1/2}}\right|<\delta$ for $k_0\le k$ since $x\in{\rm supp}\,\eta$ and $|\lambda_k|\to\infty$. 
Since $\|\sigma_i(x_k)\|\le\sup_{x\in\Omega}\|\sigma_i(x)\|$, there exists a subsequence $\{\sigma_{ik_l}\}_{l=1}^\infty$ and a constant matrix $\sigma_{i\infty}$ such that $\sigma_{ik_l}(0)=\sigma_i(x_{k_l})\to\sigma_{i\infty}~(l\to\infty)$. 
Then for $k_0\le k$ 
\begin{align*}
\|\sigma_{ik_l}(x)-\sigma_{i\infty}\|&\le\|\sigma_{ik_l}(x)-\sigma_{ik_l}(0)\|+\|\sigma_{ik_l}(0)-\sigma_{i\infty}\|\\
&\le\varepsilon+\|\sigma_{ik_l}(0)-\sigma_{i\infty}\|\\
&\to\varepsilon~(l\to\infty). 
\end{align*}
Since $\varepsilon>0$ and $x\in{\rm supp}\,\eta$ are arbitrary, we get $\|\sigma_{ik_l}-\sigma_{i\infty}\|\to0~(l\to\infty)$. 
The above calculation is also valid for $\sigma_{e}$. 
Naturally, $\sigma_{i\infty}$ and $\sigma_{e\infty}$ are positive definite constant matrices. 

We consider the weak formulation of the resolvent equation under oblique boundary condition. 
For any test functions $\phi_{i,e}\in C_0^\infty(\mathbb{R}^d)$, 
\begin{align*}
\begin{cases}
e^{i\theta_{k_l}}(w_{k_l},\phi_i)_{L^2(\Omega_{k_l})}+(\sigma_{ik_l}\nabla w_{ik_l},\nabla\phi_i)_{L^2(\Omega_{k_l})}=(t_{k_l},\phi_i)_{L^2(\Omega_{k_l})}, \\
e^{i\theta_{k_l}}(w_{k_l},\phi_e)_{L^2(\Omega_{k_l})}-(\sigma_{ek_l}\nabla w_{ek_l},\nabla\phi_e)_{L^2(\Omega_{k_l})}=(t_{k_l},\phi_e)_{L^2(\Omega_{k_l})}, \\
w_{k_l}=w_{k_l}-w_{k_l}. 
\end{cases}
\end{align*}
As $l\to\infty$, 
\begin{align*}
\begin{cases}
e^{i\theta\infty}(w,\phi_i)_{L^2(\mathbb{R}^d)}+(\sigma_{i\infty}\nabla w_i,\nabla\phi_i)_{L^2(\mathbb{R}^d)}=0, \\
e^{i\theta\infty}(w,\phi_e)_{L^2(\mathbb{R}^d)}-(\sigma_{e\infty}\nabla w_e,\nabla\phi_e)_{L^2(\mathbb{R}^d)}=0, \\
w=w_i-w_e. 
\end{cases}
\end{align*}
\end{proof}

\begin{flushleft}
\textbf{Case(4-ii)} $\Omega_\infty=\mathbb{R}_+^d$
\end{flushleft}

\begin{proposition}\label{limit2}
The limit $w, w_i, w_e\in\bigcap_{n<p<\infty}W_{loc}^{2,p}(\overline{\mathbb{R}_+^d})\cap W^{1,\infty}(\overline{\mathbb{R}_+^d})$ satisfy that for any $\phi_{i,e}\in C_0^\infty(\mathbb{R}^d)|_{\mathbb{R}_+^d}$ 
\begin{align*}
\begin{cases}
e^{i\theta\infty}(w,\phi_i)_{L^2(\mathbb{R}_+^d)}+(\sigma_{i\infty}\nabla w_i,\nabla\phi_i)_{L^2(\mathbb{R}_+^d)}=0, \\
e^{i\theta\infty}(w,\phi_e)_{L^2(\mathbb{R}_+^d)}-(\sigma_{e\infty}\nabla w_e,\nabla\phi_e)_{L^2(\mathbb{R}_+^d)}=0, \\
w=w_i-w_e, 
\end{cases}
\end{align*}
where $\theta_\infty=\lim_{k\to\infty}\theta_k$ and $\sigma_{i\infty},\sigma_{e\infty}$ are constant coefficients matrices which satisfy $\mathrm{(\ref{UE})}$ and $\mathrm{(\ref{EV})}$. 
\end{proposition}
We can prove this proposition by similar calculation to Case (4-i). 

\begin{flushleft}
\textbf{Step 5 (Uniqueness)}
\end{flushleft}

In this last step we prove that limit functions are unique. 
The method is to reduce existence of solution to dual problems and use the fundamental lemma of calculus of variation. 
In order to solve the dual problem we use the Fourier transform. 
In the half space case we extend to the whole space. 
However, we have to pay attention to the boundary condition. 
We overcome the difficulty by using the condition $\mathrm{(\ref{EV})}$. 

\begin{flushleft}
\textbf{Case(5-i)} $\Omega_\infty=\mathbb{R}^d$
\end{flushleft}

\begin{lemma}\label{weak1}
Let $w, w_i, w_e \in \bigcap_{n<p<\infty} W^{2,p}_{loc}(\mathbb{R}^d)\cap W^{1,\infty}(\mathbb{R}^d)$ satisfy 
\begin{align}\label{limiteq1}
\begin{cases}
e^{i\theta\infty}(w,\phi_i)_{L^2(\mathbb{R}^d)}+(\sigma_{i\infty}\nabla w_i,\nabla\phi_i)_{L^2(\mathbb{R}^d)}=0&(\forall\phi_i\in C_0^\infty(\mathbb{R}^d)), \\
e^{i\theta\infty}(w,\phi_e)_{L^2(\mathbb{R}^d)}-(\sigma_{e\infty}\nabla w_e,\nabla\phi_e)_{L^2(\mathbb{R}^d)}=0&(\forall\phi_e\in C_0^\infty(\mathbb{R}^d)), \\
w=w_i-w_e, \\
\end{cases}
\end{align}
then $w=0$ and $w_i=w_e=$constant. 
\end{lemma}

\begin{proof}[Proof of Lemma \ref{weak1}]
Equations (\ref{limiteq1}) implies the following equations 
\begin{align*}
\begin{cases}
\left(w_i,e^{i\theta_\infty}\phi_i-\nabla\cdot(\sigma_{i\infty}\nabla\phi_i)\right)_{L^2(\mathbb{R}^d)}-(w_e,e^{i\theta_\infty}\phi_i)_{L^2(\mathbb{R}^d)}=0, \\
(w_i,e^{i\theta_\infty}\phi_e)_{L^2(\mathbb{R}^d)}-\left(w_e,e^{i\theta_\infty}\phi_e-\nabla\cdot(\sigma_{e\infty}\nabla\phi_e)\right)_{L^2(\mathbb{R}^d)}=0, 
\end{cases}
\end{align*}
\begin{align*}
\bigl(w_i,e^{i\theta_\infty}(\phi_i+\phi_e)-&\nabla\cdot(\sigma_{i\infty}\nabla\phi_i) \bigr)_{L^2(\mathbb{R}^d)}\nonumber\\
&-\bigl(w_e,e^{i\theta_\infty}(\phi_i+\phi_e)-\nabla\cdot(\sigma_{e\infty}\nabla\phi_e)\bigr)_{L^2(\mathbb{R}^d)}=0. 
\end{align*}
Since $C_0^\infty(\mathbb{R}^d)$ is dense in $\mathcal{S}(\mathbb{R}^d)$, we can take $\phi_{i,e}$ in $\mathcal{S}(\mathbb{R}^d)$ as test functions. 
So we consider the dual problem of the limit equation. 
For all $\psi_{i,e}\in C_0^\infty(\mathbb{R}^d)$ satisfying $\int_{\mathbb{R}^d}(\psi_i-\psi_e)dx=0$, we would like to find solutions $\phi_{i,e}\in\mathcal{S}(\mathbb{R}^d)$ such that 
\begin{align*}
&e^{i\theta_\infty}(\phi_i+\phi_e)-\nabla\cdot(\sigma_{i\infty}\nabla\phi_i)=\psi_i&{\rm in}&~\mathbb{R}^d, \\
&e^{i\theta_\infty}(\phi_i+\phi_e)-\nabla\cdot(\sigma_{e\infty}\nabla\phi_e)=\psi_e&{\rm in}&~\mathbb{R}^d. 
\end{align*}
We are able to solve these equations by the Fourier transform. 
Solutions $\phi_{i,e}\in\mathcal{S}(\mathbb{R}^d)$ are of the form 
\begin{align*}
\phi_i=\mathcal{F}^{-1}\left(\frac{\left(\langle\sigma_{e\infty}\xi,\xi\rangle+e^{i\theta_\infty}\right)\mathcal{F}\psi_i-e^{i\theta_\infty}\mathcal{F}\psi_e}{\langle\sigma_{i\infty}\xi,\xi\rangle\langle\sigma_{e\infty}\xi,\xi\rangle+e^{i\theta_\infty}\left(\langle\sigma_{i\infty}\xi,\xi\rangle+\langle\sigma_{e\infty}\xi,\xi\rangle\right)}\right), \\
\phi_e=\mathcal{F}^{-1}\left(\frac{\left(\langle\sigma_{i\infty}\xi,\xi\rangle+e^{i\theta_\infty}\right)\mathcal{F}\psi_e-e^{i\theta_\infty}\mathcal{F}\psi_i}{\langle\sigma_{i\infty}\xi,\xi\rangle\langle\sigma_{e\infty}\xi,\xi\rangle+e^{i\theta_\infty}\left(\langle\sigma_{i\infty}\xi,\xi\rangle+\langle\sigma_{e\infty}\xi,\xi\rangle\right)}\right), 
\end{align*}
where $\mathcal{F}$ and $\mathcal{F}^{-1}$ denote the Fourier transform and its inverse. 
Therefore, we have for all $\psi_{i,e}\in C_0^\infty(\mathbb{R}^d)$ satisfying $\int_{\mathbb{R}^d}(\psi_i-\psi_e)dx=0$, 
\begin{align*}
(w_i,\psi_i)_{L^2(\mathbb{R}^d)}-(w_e,\psi_e)_{L^2(\mathbb{R}^d)}=0. 
\end{align*}
Let $\psi_i=\psi_e$ then $(w,\psi_i)_{L^2(\mathbb{R}^d)}=0$  for all $\psi_i\in C_0^\infty(\mathbb{R}^d)$. 
By fundamental lemma of calculus of variations, we get $w\equiv0$. 
Let $\psi_e\equiv0$ then $(w_i,\psi_i)_{L^2(\mathbb{R}^d)}=0$ for all $\psi_i\in C_0^\infty(\mathbb{R}^d)$ satisfying $\int_{\mathbb{R}^d}\psi_idx=0$. 
This means $w_i\equiv{\rm constant}$. 
Obviously $w_e=w_i$ since $w=w_i-w_e$. 
\end{proof}

\begin{lemma}\label{weak2}
Let $w, w_i, w_e \in \bigcap_{n<p<\infty} W^{2,p}_{loc}(\mathbb{R}_+^d)\cap W^{1,\infty}(\mathbb{R}_+^d)$ satisfy 
\begin{align}\label{limiteq2}
\begin{cases}
e^{i\theta\infty}(w,\phi_i)_{L^2(\mathbb{R}_+^d)}+(\sigma_{i\infty}\nabla w_i,\nabla\phi_i)_{L^2(\mathbb{R}_+^d)}=0&(\forall\phi_i\in C_0^\infty(\mathbb{R}^d)|_{\mathbb{R}_+^d}), \\
e^{i\theta\infty}(w,\phi_e)_{L^2(\mathbb{R}^d)}-(\sigma_{e\infty}\nabla w_e,\nabla\phi_e)_{L^2(\mathbb{R}_+^d)}=0&(\forall\phi_e\in C_0^\infty(\mathbb{R}^d)|_{\mathbb{R}_+^d}), \\
w=w_i-w_e, \\
\end{cases}
\end{align}
then $w=0$ and $w_i=w_e=$constant. 
\end{lemma}

\begin{proof}[Proof of lemma \ref{weak2}]
Equations (\ref{limiteq2}) implies the following equations; 
\begin{align*}
\begin{cases}
\left(w_i,e^{i\theta_\infty}\phi_i-\nabla\cdot(\sigma_{i\infty}\nabla\phi_i)\right)_{L^2(\mathbb{R}_+^d)}-(w_e,e^{i\theta_\infty}\phi_i)_{L^2(\mathbb{R}_+^d)}=0\\
\hspace{45mm}(\forall\phi_i\in C_0^\infty(\mathbb{R}^d)|_{\mathbb{R}_+^d}~{\rm s.t.}~\sigma_{i\infty}\nabla\phi_i\cdot n_\infty=0), \\
(w_i,e^{i\theta_\infty}\phi_e)_{L^2(\mathbb{R}_+^d)}-\left(w_e,e^{i\theta_\infty}\phi_e-\nabla\cdot(\sigma_{e\infty}\nabla\phi_e)\right)_{L^2(\mathbb{R}_+^d)}=0\\
\hspace{45mm}(\forall\phi_e\in C_0^\infty(\mathbb{R}^d)|_{\mathbb{R}_+^d}~{\rm s.t.}~\sigma_{e\infty}\nabla\phi_e\cdot n_\infty=0), 
\end{cases}
\end{align*}
\begin{align*}
\bigl(w_i,e^{i\theta_\infty}(\phi_i+\phi_e)-&\nabla\cdot(\sigma_{i\infty}\nabla\phi_i)\bigr)_{L^2(\mathbb{R}_+^d)}\\
&-\bigl(w_e,e^{i\theta_\infty}(\phi_i+\phi_e)-\nabla\cdot(\sigma_{e\infty}\nabla\phi_e)\bigr)_{L^2(\mathbb{R}_+^d)}=0. 
\end{align*}
The problem can be reduced to the whole space. 
Let $Ew_{i,e}$ be an even extension to the whole space $\mathbb{R}^d$, i.e. 
\begin{align*}
Ew_{i,e}(x):=
\begin{cases}
w_{i,e}(x',x_d)~&(x_d\ge0)\\
w_{i,e}(x',-x_d)~&(x_d<0). 
\end{cases}
\end{align*}
Matrices $\sigma_{i\infty}$ and $\sigma_{e\infty}$ are constant so we extend these to whole space $\mathbb{R}^d$, which we simply write by $\sigma_{i\infty}$ and $\sigma_{e\infty}$. 
Since $\sigma_{i\infty}\nabla w_i\cdot n_\infty=\nabla w_i\cdot n_\infty=0$, $\sigma_{e\infty}\nabla w_e\cdot n_\infty=\nabla w_e\cdot n_\infty=0$ and $w_{i,e}\in\bigcap_{n<p<\infty}W_{loc}^{2,p}(\overline{\mathbb{R}_+^d})\cap W^{1,\infty}(\mathbb{R}_+^d)$, we have $Ew_{i,e}\in\bigcap_{n<p<\infty}W_{loc}^{2,p}(\mathbb{R}^d)\cap W^{1,\infty}(\mathbb{R}^d)$.  
For arbitrary $\varphi_{i,e}\in C_0^\infty(\mathbb{R}^d)$, let $\varphi_{i,e}^{{\rm even}}$ and $\varphi_{i,e}^{{\rm odd}}$ be the even and odd parts of $\varphi_{i,e}$, i.e. 
\begin{align*}
\varphi_{i,e}^{{\rm even}}(x)&:=\frac{\varphi_{i,e}(x',x_d)+\varphi_{i,e}(x',-x_d)}{2}, \\
\varphi_{i,e}^{{\rm odd}}(x)&:=\frac{\varphi_{i,e}(x',x_d)-\varphi_{i,e}(x',-x_d)}{2}. 
\end{align*}
For simplicity, set a linear operator $L_i\cdot:=e^{i\theta_\infty}\cdot-\nabla\cdot(\sigma_{i\infty}\nabla\cdot)$. 
From the assumption of $\sigma_i$, note $\sigma_{i\infty}$ have the form of 
$\sigma_{i\infty}=
\left(\begin{array}{cc}
\tilde{\sigma}_{i\infty}&0\\
0&\tau_i
\end{array}\right)$
for some constant $(d-1)\times(d-1)$ matrix $\tilde{\sigma}_{i\infty}$ and $\tau_i>0$ because $(0,\cdots,0,-1)$ is eigenvector of $\sigma_{i\infty}$. 
So we have that $L_i\varphi_i^{\rm even}$ is even function and $L_i\varphi_e^{\rm odd}$ is odd function. 
Consider $L_e$ same as $L_i$. 
Naturally, $L_e$ also has the same property. 
Then we have 
\begin{align*}
&(Ew_i,e^{i\theta_\infty}\varphi_e+L_i \varphi_i)_{L^2(\mathbb{R}^d)}-(Ew_e,e^{i\theta_\infty}\varphi_i+L_e\varphi_e)_{L^2(\mathbb{R}^d)}\\
=&\left(Ew_i,e^{i\theta_\infty}(\varphi_e^{{\rm even}}+\varphi_e^{{\rm odd}})+L_i(\varphi_i^{{\rm even}}+\varphi_i^{{\rm odd}})\right)_{L^2(\mathbb{R}^d)}\\
&-\left(Ew_e,e^{i\theta_\infty}(\varphi_i^{{\rm even}}+\varphi_i^{{\rm odd}})+L_e(\varphi_e^{{\rm even}}+\varphi_e^{{\rm odd}})\right)_{L^2(\mathbb{R}^d)}\\
=&(Ew_i,e^{i\theta_\infty}\varphi_e^{{\rm even}}+L_i \varphi_i^{{\rm even}})_{L^2(\mathbb{R}^d)}-(Ew_e,e^{i\theta_\infty}\varphi_i^{{\rm even}}+L_e\varphi_e^{{\rm even}})_{L^2(\mathbb{R}^d)}\\
=&2\left\{(w_i,e^{i\theta_\infty}\varphi_e^{{\rm even}}+L_i \varphi_i^{{\rm even}})_{L^2(\mathbb{R}_+^d)}-(w_e,e^{i\theta_\infty}\varphi_i^{{\rm even}}+L_e\varphi_e^{{\rm even}})_{L^2(\mathbb{R}_+^d)}\right\}. 
\end{align*}
The function $\varphi_i^{{\rm even}}$ satisfies $\sigma_{i\infty}\nabla\varphi_i^{{\rm even}}\cdot n_\infty=\nabla\varphi_i^{{\rm even}}\cdot n_\infty=0$. 
Function $\varphi_e^{\rm even}$ also satisfies same boundary condition. 
Since the last term of above calculation equals to zero, we conclude that for any $\varphi_{i,e}\in C_0^\infty(\mathbb{R}^d)$ 
\begin{align*}
(Ew_i,e^{i\theta_\infty}\varphi_e+L_i \varphi_i)_{L^2(\mathbb{R}^d)}-(Ew_e,e^{i\theta_\infty}\varphi_i+L_e\varphi_e)_{L^2(\mathbb{R}^d)}=0. 
\end{align*}
This means $Ew_i=Ew_e={\rm constant}$ by the Case(4-i). 
 Therefore we have $w=0$ and $w_i=w_e={\rm constant}$. 
\end{proof}

Results of Step 3 and Step 5 are contradictory, so the proof of Theorem \ref{main result} is now complete. 
\end{proof}

\section{Bidomain operators}

\subsection{Definition of bidomain operators in $L^p$ spaces}

In this subsection we define bidomain operators in $L^p$ spaces for $1<p<\infty$. 
To avoid technical difficulties we assume that $\Omega$ is a bounded $C^2$-domain. 
We reformulate resolvent equations corresponding to the parabolic and elliptic system as are derived in \cite{BCP}. 
The new system contains only $u$ and $u_e$ as unknown functions. 
Since $u_i=u+u_e$ by (\ref{action potential}), the new system is of the form: 
\begin{align}
&\lambda u-\nabla\cdot(\sigma_i\nabla u)-\nabla\cdot(\sigma_i\nabla u_e)=s & \mathrm{in}&~\Omega,\label{parabolic} \\
&-\nabla\cdot(\sigma_i\nabla u+(\sigma_i+\sigma_e)\nabla u_e)=0 & \mathrm{in}&~\Omega,\label{elliptic} \\
&\sigma_i\nabla u\cdot n+\sigma_i\nabla u_e\cdot n=0 & \mathrm{on}&~\partial\Omega, \\
&\sigma_i\nabla u\cdot n+(\sigma_i+\sigma_e)\nabla u_e\cdot n=0 & \mathrm{on}&~\partial\Omega. \label{boundary condition} 
\end{align}

Let $1<p<\infty$ and $\Omega$ be a bounded $C^2$-domain. 
Set $L^p_{av}(\Omega):=\{u\in L^p(\Omega)\mid\int_\Omega u dx=0\}$ and the operator $P_{av}$ defined by $P_{av}u:=u-\frac{1}{|\Omega|}\int_\Omega u dx$, which is the orthogonal projection. 
Evidently, $L^p_{av}(\Omega)$ is a closed subspace in $L^p(\Omega)$ and $P_{av}$ is a bounded linear operator on $L^p(\Omega)$. 
We similarly define a function space $W^{2,p}_{av}(\Omega)$, i.e. $W^{2,p}_{av}(\Omega)=W^{2,p}(\Omega)\cap L^p_{av}(\Omega)$.  
We define an operator $A_{i,e}$ in $L^p_{av}(\Omega)$ with the domain $D(A_{i,e})$ corresponding to a uniformly elliptic operator $-\nabla\cdot(\sigma_{i,e}\nabla\cdot)$ with the oblique boundary condition. 
It is explicitly defined as 
\begin{align*}
&u\in D(A_{i,e}):=\left\{u\in W^{2,p}_{av}(\Omega)\mid\sigma_{i,e}\nabla u\cdot n=0 {\rm~a.e.~in~}\partial\Omega\right\}\subset L^p_{av}(\Omega), \\
&A_{i,e}u:=-\nabla\cdot(\sigma_{i,e}\nabla u). 
\end{align*}

\begin{lemma}[\cite{sim}]\label{inverse op}
Let $1<p<\infty$ and let $\Omega$ be a bounded $C^2$-domain. 
Assume that $\sigma_{i,e}\in C^1(\overline{\Omega})$ satisfies {\rm(\ref{UE})}. 
Then  the operator $A_i$ is densely defined closed linear operator on $L^p_{av}(\Omega)$ and for any $f\in L^p_{av}(\Omega)$ there uniquely exists $u\in D(A_i)$ such that $A_i u=f$. 
The operator $A_e$ also has the same property. 
\end{lemma}

If we assume that $\sigma_{i,e}\nabla u\cdot n=0$ is equivalent to $\nabla u\cdot n=0$, then $D(A_i)=\left\{u\in W^{2,p}_{av}(\Omega)\mid\nabla u\cdot n=0 {\rm~a.e.~in~}\partial\Omega\right\}=D(A_e)$. 
So we are able to define the operator $A_i+A_e$ with the domain $D(A_i)(=D(A_e))$ and we observe that inverse operator $(A_i+A_e)^{-1}$ on $L^p_{av}$ is a bounded linear operator. 
Under $\int_\Omega u_e dx=0$, which is often used assumption to study bidomain equations, from (\ref{elliptic}), 
\begin{align*}
&A_iP_{av}u+(A_i+A_e)u_e=0\\
\Leftrightarrow&(A_i+A_e)u_e=-A_iP_{av}u\hspace{5mm}(\in L^p_{av}(\Omega))\\
\Leftrightarrow&u_e=-(A_i+A_e)^{-1}A_i P_{av}u\hspace{1mm}(\in D(A_i)). 
\end{align*}
We substitute this into (\ref{parabolic}) to set 
\begin{flalign*}
&\lambda u+A_i P_{av}u-A_i(A_i+A_e)^{-1}A_i P_{av}u=s\\
\Leftrightarrow&\lambda u+A_i(A_i+A_e)^{-1}A_e P_{av}u=s. 
\end{flalign*}
We are ready to define bidomain operators $A$. 

\begin{definition}[{\cite[Definition 12($p=2$)]{BCP}}]\label{bidomain op}
For $1<p<\infty$, we define the bidomain operator $A: D(A):=\{u\in W^{2,p}(\Omega)\mid\nabla u\cdot n=0 {\rm~a.e.~in~}\partial\Omega\}\subset L^p(\Omega)\to L^p(\Omega)$ by 
\begin{align}
A=A_i(A_i+A_e)^{-1}A_eP_{av}. 
\end{align}
\end{definition}

Under $\int_\Omega u_e dx=0$, equations (\ref{parabolic})-(\ref{boundary condition}) for the function $u$ can be written in a single resolvent equation of the form 
\begin{align}
(\lambda+A)u=s\hspace{5mm}{\rm in}~\Omega\label{operator eq}. 
\end{align}
Once we solve this equation, we are able to derive $u_e=-(A_i+A_e)^{-1}A_iP_{av}u$. 

\subsection{Resolvent set of bidomain operators}

We study existence and uniqueness of the solution for bidomain equations ($\ref{operator eq}$). 
We derive $W^{2,p}$ a priori estimate for fixed $\lambda$ by $W^{2,p}$ a priori estimate for the usual elliptic operator $A_e$. 
To define the bidomain operator $A$, we now assume that $\Omega$ is a bounded $C^2$-domain and $\sigma_{i,e}\in C^1(\overline{\Omega})$ satisfy {\rm(\ref{UE})} and {\rm(\ref{EV})}, which will be used throughout. 

\begin{theorem}[A priori estimate for bidomain operators]\label{a priori}
Let $1<p<\infty$. 
For each $\lambda\in\Sigma_{\pi,0}$ there exists $C_\lambda>0$ such that 
\begin{align*}
\|u\|_{W^{2,p}(\Omega)}\le C_\lambda\left(\|(\lambda+A)u\|_{L^p(\Omega)}+\|u\|_{L^p(\Omega)}\right) 
\end{align*}
for all $u\in D(A)$. 
\end{theorem}

\begin{proof}
We operate $(A_i+A_e)A_i^{-1}P_{av}$ to $(\lambda +A)u=s$ to get $(\lambda+A_e)P_{av}u = (A_i+A_e)A_i^{-1}P_{av}s-\lambda A_eA_i^{-1}P_{av}u$. 
Since $A_e$ has a resolvent estimate \cite{Tan1}, for each $\varepsilon\in(0,\pi/2)$ there exists $C>0$ such that 
\begin{align*}
&|\lambda|\|P_{av}u\|_{L^p(\Omega)}+|\lambda|^{1/2}\|\nabla P_{av}u\|_{L^p(\Omega)}+\|\nabla^2 P_{av}u\|_{L^p(\Omega)}\\
&\le C\|(A_i+A_e)A_i^{-1}P_{av}s-\lambda A_eA_i^{-1}P_{av}u\|_{L^p(\Omega)}\\
&\le C\|s\|_{L^p(\Omega)}+C|\lambda|\|u\|_{L^p(\Omega)}\cdots(\ast\ast)
\end{align*} 
for all $\lambda\in\Sigma_{\pi-\varepsilon,0}$. 
Here, note that $(A_i+A_e)A_i^{-1}P_{av}$ and $A_eA_i^{-1}P_{av}$ are bounded operators in $L^p(\Omega)$. 
From above inequality we have for any $\lambda\in\Sigma_{\pi,0}$ there exists $C_\lambda>0$ independent of $u$ (may depend on $\lambda$) such that 
\begin{align*}
\|u\|_{W^{2,p}(\Omega)}\le C_\lambda\left(\|(\lambda+A)u\|_{L^p(\Omega)}+\|u\|_{L^p(\Omega)}\right). 
\end{align*}
\end{proof}

By this theorem we observe that the bidomain operator $A$ in $L^p$ spaces is a densely defined closed linear operator. 

Let $A_p$ be the bidomain operator in $L^p$ spaces. 
We  characterize the resolvent set of bidomain operator $A_p$ in $L^p$ spaces from the previous result \cite{BCP} that the bidomain operator $A_2$ is non-negative self-adjoint operator in $L^2$ spaces, i.e. $\Sigma_{\pi,0}\subset\rho(-A_2)$. 

\begin{lemma}\label{W^{2,p} estimate}
Let $1<p<\infty$. 
Let $\lambda\in\Sigma_{\pi,0}$. 
Assume that $(\lambda+A_p)u=0$ implies $u=0$, then the inequality $\|u\|_{W^{2,p}(\Omega)}\le C_\lambda\|(\lambda+A_p)u\|_{L^p(\Omega)}$ holds, where $C_\lambda>0$ is the constant independent of $u\in D(A_p)$. 
\end{lemma}

\begin{proof}
We argue by contradiction. 
If the inequality were false, there would exist a sequence $\{u_k\}_{k=1}^\infty\subset D(A_p)$ satisfying 
\begin{align*}
\|u_k\|_{W^{2,p}(\Omega)}=1,~~\|(\lambda+A_p)u_k\|_{L^p(\Omega)}<1/k. 
\end{align*}
By the compactness of the imbedding $W^{2,p}(\Omega)\to W^{1,p}(\Omega)$ (Rellich's compactness theorem), there exists a subsequence $\{u_{k_l}\}_{l=1}^\infty$ converging strongly in $W^{1,p}(\Omega)$ to a function $u\in D(A_p)$. 
Define $\tilde{u}_{k_l}=(A_i+A_e)^{-1}A_eP_{av}u_{k_l}$, $\tilde{u}=(A_i+A_e)^{-1}A_eP_{av}u$ and the conjugate exponent $p'$ of $p$, $\frac{1}{p}+\frac{1}{p'}=1$ for $1<p<\infty$. 
We have $\{\tilde{u}_{k_l}\}_{l=1}^\infty$ are uniform bounded in $W^{2,p}(\Omega)$ converging to a function $\tilde{u}\in D(A_p)$. 
Since 
\begin{align*}
\int_\Omega \lambda u_{k_l}v + \sigma_i\nabla\tilde{u}_{k_l}\cdot \nabla v\to \int_\Omega \lambda uv+\sigma_i \nabla\tilde{u}\cdot\nabla v
\end{align*}
for all $v\in L^{p'}(\Omega)$, we must have $\int_\Omega \lambda uv+\sigma_i \nabla \tilde{u}\cdot\nabla v=0$ for all $v\in L^{p'}(\Omega)$. 
Hence $(\lambda+A_p)u=0$. 
The uniqueness implies $u=0$. 
However, the estimate in Theorem \ref{a priori} implies 
\begin{align*}
1=\|u_{k_l}\|_{W^{2,p}(\Omega)}\le C\left(\|(\lambda+A_p)u_{k_l}\|_{L^p(\Omega)}+\|u_{k_l}\|_{L^p(\Omega)}\right). 
\end{align*}
Sending $l\to\infty$ implies $1\le C\liminf_{k\to\infty}\|u_{k_l}\|_{L^p(\Omega)}$. 
This would contradict that $u_{k_l}\to u=0$ strongly in $W^{1,p}(\Omega)$. 
\end{proof}

\begin{theorem}\label{resolvent set}
Let $2\le p<\infty$. 
Then for any $\lambda\in\Sigma_{\pi,0}$ and $s\in L^p(\Omega)$, there uniquely exists $u\in D(A_p)$ such that $(\lambda+A_p)u=s$. 
\end{theorem}

\begin{proof}
If $\lambda\in\Sigma_{\pi,0}$ and $u\in D(A_p)$ satisfy $(\lambda+A_p)u=0$ then $u=0$ since $u\in D(A_p)\subset D(A_2)$ and $\lambda\in\rho(-A_2)$. 
For existence of a solution to a bidomain equation we use the continuity method \cite{GT}. 
For each $t\in[0,1]$ we set 
\begin{align*}
L_t:=\lambda+A_i(tA_i+A_e)^{-1}A_eP_{av}:D(A_p)\to L^p(\Omega). 
\end{align*}
By Lemma \ref{W^{2,p} estimate} we see there is a constant $C_\lambda>0$ such that $\|u\|_{W^{2,p}(\Omega)}\le C_\lambda\|L_t u\|_{L^p(\Omega)}$ for all $u\in D(A_p)$ and $t\in[0,1]$. 
Suppose that $L_{\tilde{t}}:D(A_p)\to L^p(\Omega)$ is onto for some $\tilde{t}\in[0,1]$, then 
$L_{\tilde{t}}$ is one-to-one. 
Hence there exists inverse mapping $L_{\tilde{t}}^{-1}: L^p(\Omega)\to D(A_p)$. \\
For $t\in[0,1]$ and $s\in L^p(\Omega)$, the equation $L_t u=s$ is equivalent to the equation
\begin{align*}
L_tu&=s\\
L_{\tilde{t}}u&=s+(L_{\tilde{t}}-L_t)u \\
&=s+(t-\tilde{t})A_i(\tilde{t}A_i+A_e)^{-1}A_i(tA_i+A_e)^{-1}A_eP_{av}u \\
u&=L_{\tilde{t}}^{-1}\{s+(t-\tilde{t})A_i(\tilde{t}A_i+A_e)^{-1}A_i(tA_i+A_e)^{-1}A_eP_{av}u\}. 
\end{align*}
Set the mapping $T:D(A_p)\to D(A_p)$ and $\delta>0$ of the form 
\begin{align*}
Tu&=L_{\tilde{t}}^{-1}\{s+(t-\tilde{t})A_i(\tilde{t}A_i+A_e)^{-1}A_i(tA_i+A_e)^{-1}A_eP_{av}u\}, \\
\delta&=\left\{\sup_{t,\tilde{t}\in[0,1]}\|L_{\tilde{t}}^{-1}\{A_i(\tilde{t}A_i+A_e)^{-1}A_i(tA_i+A_e)^{-1}A_eP_{av}\|_{\mathcal{L}(W^{2,p}(\Omega))}\right\}^{-1}. 
\end{align*}
The mapping $T$ is a contraction mapping if $|t-\tilde{t}|<\delta$ and hence the mapping $L_t:D(A_p)\to L^p(\Omega)$ is onto for all $t\in[0,1]$ satisfying $|t-\tilde{t}|<\delta$ because of $\delta$ is independent of $t$, $\tilde{t}$. 
By dividing the interval $[0,1]$ into subintervals of length less than $\delta$, we see that the mapping $L_t$ is onto for all $t\in[0,1]$ because of $L_0=\lambda+A_iP_{av}:D(A_p)\to L^p(\Omega)$ is onto when $\lambda\in\Sigma_{\pi,0}$. 
\end{proof}

\begin{lemma}
Let $1<p<\infty$. 
The adjoint of the bidomain operator $A_p$ is $A_{p'}$. 
\end{lemma}

\begin{proof}
Let $u\in D(A_p)$, $v\in D(A_{p'})$ and $2\le p<\infty$. 
For simplicity, we write $\langle\cdot,\cdot\rangle:=\langle\cdot,\cdot\rangle_{L^p(\Omega)\times L^{p'}(\Omega)}$. 
\begin{align*}
&\langle A_p u,v\rangle\\
&=\langle A_iP_{av}u-A_i(A_i+A_e)^{-1}A_iP_{av}u,v\rangle\\
&=\langle A_iP_{av}u-A_i(A_i+A_e)^{-1}A_iP_{av}u,v\rangle\\
&\hspace{15mm}-\langle A_iP_{av}u-(A_i+A_e)(A_i+A_e)^{-1}A_iP_{av}u,(A_i+A_e)^{-1}A_iP_{av}v\rangle\\
&=\langle A_iP_{av}u-A_i(A_i+A_e)^{-1}A_iP_{av}u, v-(A_i+A_e)^{-1}A_iP_{av}v\rangle\\
&\hspace{15mm}+\langle A_e(A_i+A_e)^{-1}A_iP_{av}u,(A_i+A_e)^{-1}A_iP_{av}v\rangle\\
&=\langle u-(A_i+A_e)^{-1}A_iP_{av}u,A_iP_{av}v-A_i(A_i+A_e)A_iP_{av}v\rangle\\
&\hspace{15mm}+\langle(A_i+A_e)^{-1}A_iP_{av}u,A_e(A_i+A_e)^{-1}A_iP_{av}v\rangle\\
&=\langle u,A_{p'}v\rangle\\
&\hspace{3mm}-\langle(A_i+A_e)^{-1}A_iP_{av}u,A_iP_{av}v-A_i(A_i+A_e)^{-1}A_iP_{av}v-A_e(A_i+A_e)^{-1}A_iP_{av}v\rangle\\
&=\langle u,A_{p'}v\rangle. 
\end{align*}
So we get $A_p\subset A_{p'}^\ast$. 
In order to show $D(A_p)\supset D(A_{p'}^\ast)$, we first show that $\lambda\in\rho(-A_p)$ implies $\lambda\in\rho(-A_{p'}^\ast)$. 
Remark that $D(A_2)\subset D(A_{p'})$ and $A_{p'}u=A_2u~(u\in D(A_2))$. 
For $\lambda\in\rho(-A_p)$, $(\lambda+A_{p'})D(A_{p'})\supset(\lambda+A_{p'})D(A_2)=(\lambda+A_2)D(A_2)=L^2(\Omega)$.  
So $R\left(\lambda+A_{p'}\right)$ is dense in $L^{p'}(\Omega)$. 
Therefore $\lambda+A_{p'}^\ast$ is one-to-one in $L^p(\Omega)$. 
Since $A_p\subset A_{p'}^\ast$ and $\lambda+A_p$ is surjection in $L^p(\Omega)$, we get $\lambda+A_{p'}^\ast$ is surjection. 
This means $\lambda\in\rho(-A_{p'}^\ast)$. \\
Take $u\in D(A_{p'}^\ast)$ and for some $\lambda\in\rho(-A_p)\cap\rho(-A_{p'}^\ast)\neq\emptyset$, then 
\begin{align*}
v&:=(\lambda+A_p)^{-1}(\lambda+A_{p'}^\ast)u\in D(A_p)\\
(\lambda+A_p)v &=(\lambda+A_{p'}^\ast)u\\
(\lambda+A_{p'}^\ast)v&=(\lambda+A_{p'}^\ast)u\\
v&=u. 
\end{align*}
Therefore $D(A_{p'}^\ast)\subset D(A_p)$ and $A_{p'}^\ast=A_p$. 
Since $A_{p'}$ is a closed linear operator, we have $A_{p'}=A_{p'}^{\ast\ast}=A_p^\ast$. 
This means for all $1<p<\infty$ the adjoint of the bidomain operator $A_p$ is $A_{p'}$. 
\end{proof}
\noindent
So we have for all $1<p<\infty$, $\rho(-A_p)=\rho(-A_p^\ast)=\rho(-A_{p'})=\Sigma_{\pi,0}$. 

Our Theorem \ref{resolvent set} implies existence and uniqueness of the resolvent bidomain equation since it is equivalent to the equation (\ref{operator eq}). 

\begin{theorem}[Existence and Uniqueness]\label{EU}
Let $1<p<\infty$, $\Omega$ be a bounded $C^2$-domain and $\sigma_{i,e}\in C^1(\overline{\Omega},\mathbb{S}^d)$ satisfy {\rm (\ref{UE})} and {\rm (\ref{EV})}. 
Then for any $\lambda\in\Sigma_{\pi,0}$, $s\in L^p(\Omega)$, the resolvent problem 
\begin{equation*}
\begin{cases}
\lambda u-\nabla\cdot(\sigma_i\nabla u_i)=s\hspace{5mm}&{\rm in} \ \Omega,\\
\lambda u+\nabla\cdot(\sigma_e\nabla u_e)=s\hspace{5mm}&{\rm in} \ \Omega, \\
u=u_i-u_e\hspace{5mm}&{\rm in} \ \Omega, \\
\sigma_i\nabla u_i\cdot n=0,~\sigma_e\nabla u_e\cdot n=0\hspace{5mm}&{\rm on} \ \partial\Omega, 
\end{cases}
\end{equation*}
has a unique solution $u,u_{i,e}\in W^{2,p}(\Omega)$ satisfying $\int_\Omega u_e dx=0$. 
\end{theorem}

\subsection{Analyticity of semigroup generated by bidomain operators}

We will study bidomain equations in the framework of an analytic semigroup, so let us recall the definition of a sectorial operator. 
Let $X$ be a complex Banach space and $A:D(A)\subset X\to X$ be a linear operator, may not have a dense domain. 

\begin{definition}
The operator $A$ is said to be a sectorial operator with angle $\theta(\in[0,\pi/2)$ if for each $\varepsilon \in(0,\pi/2)$ there exist $C>0$ and $M\ge0$ such that
\begin{align*}
(1)~\rho(-A)\supset\Sigma_{\pi-\theta,M},~~(2)~\sup_{\lambda\in\Sigma_{\pi-\theta-\varepsilon,M}}|\lambda|\|(\lambda+A)^{-1}\|_{\mathcal{L}(X)}\le C. 
\end{align*}
\end{definition}

We do not assume that the operator $A$ has a dense domain. 
So it may happen that the analytic semigroup $\{e^{-tA}\}_{t\ge0}$ generated by the operator $A$ may not be strongly continuous, that is for each $x\in X$ the function $t\mapsto e^{-tA}x$ is not necessarily continuous on $[0,\infty)$. 
We call $\{e^{-tA}\}_{t\ge0}$ $C_0$-analytic semigroup if for each $x\in X$, $t\mapsto e^{-tA}x$ is continuous on $[0,\infty)$. 
We have that if the operator $A$ is a sectorial operator with angle $\theta$, then $t\mapsto e^{-tA}$ is analytic in $[0,\infty)$ and it can be extended holomorphically in a sector with opening angle $2(\pi/2-\theta)$. 
For sectorial operators, it is known that 
\begin{align*}
\{e^{-tA}\}_{t\ge0}:{\rm strongly~continuous} \Leftrightarrow\forall x\in X, \lim_{t\to0}e^{-tA}x=x\Leftrightarrow\overline{D(A)}=X. 
\end{align*}
Therefore, $\{e^{-tA}\}_{t\ge0}$ is $C_0$-analytic semigroup if and only if the operator $A$ is a sectorial operator with dense domain $D(A)$ in $X$ (See \cite{Lun}). 

Let us go back to consider bidomain operators. 
Note that \cite{BCP} showed the bidomain operator $A$ is a non-negative self-adjoint operator in $L^2(\Omega)$ so that it is a sectorial operator. 
Namely, $\rho(-A_2)\supset\Sigma_{\pi,0}$ and for each $\varepsilon\in(0,\pi/2)$ there exists $C>0$ such that 
\begin{align*}
\sup_{\lambda\in\Sigma_{\pi-\varepsilon,0}}|\lambda|\|u\|_{L^2(\Omega)}\le C\|s\|_{L^2(\Omega)}
\end{align*}
for all $s\in L^2(\Omega)$. 
We derived an $L^\infty$ resolvent estimate ({\rm Theorem} \ref{main result}); 
for each $\varepsilon \in(0,\pi/2)$ there exist $C>0$ and $M\ge0$ such that $\rho(-A)\supset\Sigma_{\pi,M}$ 
\begin{align*}
\sup_{\lambda\in\Sigma_{\pi-\varepsilon,M}}|\lambda|\|u\|_{L^\infty(\Omega)}\le C\|s\|_{L^\infty(\Omega)}
\end{align*}
 and for all $s\in L^\infty(\Omega)$. 
 
By using Riesz-Thorin interpolation theorem, we are able to derive an $L^p$ resolvent estimate, i.e. for each $\varepsilon\in(0,\pi/2)$ and $2\le p\le\infty$ there exist $C>0$ and $M\ge0$ such that $\rho(-A_p)\supset\Sigma_{\pi,M}$ and that 
\begin{align*}
\sup_{\lambda\in\Sigma_{\pi-\varepsilon,M}}|\lambda|\|u\|_{L^p(\Omega)}\le C\|s\|_{L^p(\Omega)}
\end{align*}
 and for all $s\in L^p(\Omega)$, 
 
For $2\le p<\infty$ and its conjugate exponent $p'(\in(1,2])$, we have 
\begin{align*}
\|(\lambda+A_{p'})^{-1}\|_{\mathcal{L}(L^{p'}(\Omega))}=\|((\lambda+A_p)^{-1})^\ast\|_{\mathcal{L}(L^{p'}(\Omega))}=\|(\lambda+A_p)^{-1}\|_{\mathcal{L}(L^p(\Omega))}\le\frac{C}{|\lambda|}. 
\end{align*}

We derived the resolvent estimate for bidomain operators $-A_p$ in $L^p$ spaces for the sufficiently large $\lambda$. 
However, in the next theorem, we estimate the resolvent for all $\lambda\in\Sigma_{\pi-\varepsilon,0}$ and higher order derivatives $\|\nabla u\|_{L^p(\Omega)}$ and $\|\nabla^2 u\|_{L^p(\Omega)}$, which is similar to an elliptic operator in $L^p$ spaces. 

\begin{theorem}[$L^p$ resolvent estimates for bidomain operators]\label{full}
Let $1<p<\infty$. 
For each $\varepsilon\in(0,\pi/2)$ there exists $C>0$ depending only on $\varepsilon$ such that the unique solution $u\in D(A_p)$ of the resolvent equation $(\lambda+A_p)u=s$ satisfies  
\begin{align*}
|\lambda|\|u\|_{L^p(\Omega)}+|\lambda|^{1/2}\|\nabla u\|_{L^p(\Omega)}+\|\nabla^2 u\|_{L^p(\Omega)}\le C\|s\|_{L^p(\Omega)}
\end{align*}
for all $\lambda \in\Sigma_{\pi-\varepsilon,0}$ and $s\in L^p(\Omega)$. 
\end{theorem}

\begin{proof}
We divide the resolvent estimate $(\lambda+A_p)u=s$ into $(\lambda+A_p)u_1=P_{av}s$ and $(\lambda+A_p)u_2=s-P_{av}s$. 
Note that $u=u_1+u_2$, $P_{av}s\in L^p_{av}(\Omega)$, $s-P_{av}s$ is a constant and the origin $0$ belongs to $\rho(-A_p|_{L^p_{av}(\Omega)})$. 
For each $\varepsilon \in(0,\pi/2)$ we fix $M\ge0$ which is the constant in the above explanation. 
Since $(\lambda+A_p)^{-1}P_{av}s=(\lambda+A_p|_{L^p_{av}(\Omega)})^{-1}P_{av}s$ and the resolvent operator $(\lambda+A_p|_{L^p_{av}(\Omega)})^{-1}$ is uniform bounded in a compact subset $\overline{\Sigma_{\pi-\varepsilon,0}}\cap \overline{B(0,2M)}$, we have there exists $C>0$ depending on $\varepsilon$ such that 
\begin{align*}
\|u_1\|_{L^p(\Omega)}&=\|(\lambda+A_p)^{-1}P_{av}s\|_{L^p(\Omega)}\\
&=\|(\lambda+A_p|_{L^p_{av}(\Omega)})^{-1}P_{av}s\|_{L^p(\Omega)}\\
&\le \frac{C}{|\lambda|+1}\|P_{av}s\|_{L^p(\Omega)}
\end{align*}
for all $\lambda\in\Sigma_{\pi-\varepsilon,0}\cap B(0,2M)$. 
On the other hand we have $u_2=\frac{1}{\lambda}(s-P_{av}s)$, so there exists $C>0$ such that 
\begin{align*}
\|u_2\|_{L^p(\Omega)}&=\|\frac{1}{\lambda}(s-P_{av}s)\|_{L^p(\Omega)}\\
&\le \frac{C}{|\lambda|}\|s-P_{av}s\|_{L^p(\Omega)}
\end{align*}
for all $\lambda\in\Sigma_{\pi-\varepsilon,0}$. 
We use the operator $P_{av}$ is a bounded linear operator and combine two estimates. 
We have that there exists $C>0$ such that $\|u\|_{L^p(\Omega)}\le \frac{C}{|\lambda|}\|s\|_{L^p(\Omega)}$ for all $\lambda\in\Sigma_{\pi-\varepsilon,0}\cap B(0,2M)$. 
Since we have already proved the resolvent estimate for $|\lambda|>M$, the resolvent estimate holds for all $\lambda\in\Sigma_{\pi-\varepsilon,0}$. 
Estimates for higher order derivatives it follows from the key estimate $(\ast\ast)$ of the proof of Theorem \ref{a priori}. 
\end{proof}

We can also define the bidomain operator in $L^\infty(\Omega)$. 
When the domain $\Omega$ is bounded, $L^\infty(\Omega)$ is contained in $\bigcap_{n<p<\infty}L^p(\Omega)$. 
So for all $s\in L^\infty(\Omega)$ we can take a unique solution of (\ref{parabolic})-(\ref{boundary condition}) $u,u_{i,e}\in\bigcap_{n<p<\infty} W^{2,p}(\Omega)$ satisfying $\int_\Omega u_e dx=0$. 
Here, note that we cannot expect a $W^{2,\infty}(\Omega)$ solution such as a usual elliptic problem. 

For $\lambda\in\Sigma_{\pi-\varepsilon,M}$ let $R_\infty(\lambda)$ be the solution operator from $s\in L^\infty(\Omega)$ to $u\in\bigcap_{n<p<\infty}W^{2,p}(\Omega)(\subset L^\infty(\Omega))$ such that $u$ is a solution of the resolvent bidomain equation (\ref{parabolic})-(\ref{boundary condition}). 
We warn that the abstract equation (\ref{operator eq}) is not available for $L^\infty$ at this moment. 
The operator $R_\infty(\lambda)$ is a bounded operator whose operator norm is dominated by $C/|\lambda |$, i.e., 
\begin{align*}
\|R_\infty(\lambda)s\|_{L^\infty(\Omega)}\le \frac{C}{|\lambda|}\|s\|_{L^\infty(\Omega)}. 
\end{align*}
The operator $R_\infty(\lambda)$ may be regarded as a bijection operator from $L^\infty(\Omega)$ to $R_\infty(\lambda)L^\infty(\Omega)$. 
The operator $R_\infty:\Sigma_{\pi-\varepsilon,M}\to \mathcal{L}(L^\infty(\Omega))$ satisfy the following resolvent equation; 
\begin{align*}
R_\infty(\lambda)-R_\infty(\mu)=(\mu-\lambda)R_\infty(\lambda)R_\infty(\mu)~~(\lambda,\mu\in \Sigma_{\pi-\varepsilon,M}). 
\end{align*}
Namely the operator $R_\infty(\lambda)$ is a pseudo-resolvent. 
We use the following proposition. 
\begin{proposition}[{\cite[Proposition B.6.]{ABHN}}]
Set a subset $U\subset\mathbb{C}$ and a Banach space $X$. 
Let a function $R:U\to\mathcal{L}(X)$ be a pseudo-resolvent. 
Then \\
{\rm(a)} ${\rm Ker}\,R(\lambda)$ and ${\rm Ran}\,R(\lambda)$ are independent of $\lambda\in U$. \\
{\rm(b)} There is an operator $A$ on $X$ such that $R(\lambda)=(\lambda+A)^{-1}$ for all $\lambda\in U$ if and only if ${\rm Ker}\, R(\lambda)=\{0\}$. 
\end{proposition}

By this proposition, there exists a operator $A_\infty$ with the domain $D(A_\infty)=R_\infty(\lambda)L^\infty(\Omega)(\subset\cap_{n<p<\infty}W^{2,p}(\Omega))$ such that $(\lambda+A_\infty)^{-1}s=u$, i.e. $(\lambda+A_\infty)u=s$. 
We call $A_\infty$ the bidomain operator in $L^\infty(\Omega)$. 
We have the bidomain operator $A_\infty$ in $L^\infty(\Omega)$ is a sectorial operator. 
However, it is easy to see that $D(A_\infty)$ is not dense. 
Indeed, $\bigcap_{n<p<\infty}W^{2,p}(\Omega)\subset C(\overline{\Omega})$ and hence $\overline{D(A_\infty)}^{L^\infty(\Omega)}\subset C(\overline{\Omega})$, where $\overline{D(A_\infty)}^{L^\infty(\Omega)}$ is the closure of $D(A_\infty)$ in the $L^\infty(\Omega)$ norm. 
Since $C(\overline{\Omega})$  is not dense in $L^\infty(\Omega)$, $D(A_\infty)$ is not dense in $L^\infty(\Omega)$. 
We restrict the dense domain $\overline{D(A_\infty)}^{L^\infty(\Omega)}$. 
We also have $\overline{D(A_\infty)}^{L^\infty(\Omega)}=\{u\in UC(\overline{\Omega})\mid\nabla u\cdot n =0\}$, where $UC(\overline{\Omega})$ denotes the space of all the uniformly continuous functions in $\overline{\Omega}$ (see \cite{Lun}). 
So we consider again such that 
\begin{align*}
&D(\tilde{A}_\infty):=\{u\in D(A_\infty)\mid A_\infty u\in UC(\overline{\Omega})\}, \\
&\tilde{A}_\infty u:=A_\infty u. 
\end{align*}
Then the operator $\tilde{A}_\infty$ is a densely defined sectorial operator in $UC(\overline{\Omega})$. 
Our resolvent estimates (Theorem \ref{full} for $L^p$, Theorem \ref{main result} for $L^\infty$) yields the following theorem. 

\begin{theorem}[Analyticity of bidomain operators]
For $1<p<\infty$ bidomain operators $A_p$ in $L^p(\Omega)$ generate $C_0$-analytic semigroups with angle $\pi/2$. 
The operator $A_\infty$ generates a non-$C_0$-analytic semigroup with angle $\pi/2$ in $L^\infty(\Omega)$, and the operator $\tilde{A}_\infty$ generates a $C_0$-analytic semigroup with angle $\pi/2$ in $UC(\overline{\Omega})$. 
\end{theorem}

\section{Strong solutions in $L^p$ spaces}

By discussion in the previous section, we are able to study nonstationary state bidomain equations by using the bidomain operator $A$. 
Let us state the definition of a strong solution. 
Assume that $\Omega$ is a bounded $C^2$-domain, $1<p<\infty$, $s_{i,e}\in C^\nu_{loc}([0,\infty);L^p(\Omega))$ (for some $0<\nu<1$) such that $s_i(t)+s_e(t)\in L^p_{av}(\Omega)~(\forall t\ge0)$ and $f:\mathbb{R}\times\mathbb{R}^m\to\mathbb{R}^m$ and $g:\mathbb{R}\times\mathbb{R}^m\to\mathbb{R}^m$ are locally Lipschitz continuous functions. 
Before giving the definition of a strong solution, we recall parabolic-elliptic type bidomain equations. 
\begin{align}
&\partial_{t}u+f(u,w)-\nabla\cdot(\sigma_{i}\nabla u)-\nabla\cdot(\sigma_i\nabla u_e)=s_{i}&\mathrm{in}&~(0,\infty)\times\Omega\label{inner2},\\
&-\nabla\cdot(\sigma_i\nabla u+(\sigma_i+\sigma_{e})\nabla u_{e})=s_i+s_e&\mathrm{in}&~(0,\infty)\times\Omega\label{exter2},\\
&\partial_{t}w+g(u,w)=0&\mathrm{in}&~(0,\infty)\times\Omega\label{ODE2},\\
&\sigma_{i}\nabla u\cdot n+\sigma_i\nabla u_e\cdot n=0&\mathrm{on}&~(0,\infty)\times\partial\Omega\label{boundary2},\\
&\sigma_{i}\nabla u\cdot n+(\sigma_i+\sigma_e)\nabla u_e\cdot n=0&\mathrm{on}&~(0,\infty)\times\partial\Omega\label{boundary3},\\
&u(0)=u_{0},~w(0)=w_{0}&\mathrm{in}&~\Omega\label{initial data2}.
\end{align}

\begin{definition}[{\cite[Definition 18]{BCP}} Strong solution]{\label{strong sol}}
{\rm
For $\tau>0$ consider the functions $z:t\in[0,\tau)\mapsto z(t)=(u(t),w(t))\in Z:=L^p(\Omega)\times B^m~(B=L^\infty(\Omega)$ or $C^\nu(\Omega)$) and $u_e:t\in[0,\tau)\mapsto u_e(t)\in L^p(\Omega)$. 
Given $z_0=(u_0,w_0)\in Z$, we say that $(u,u_e,w)$ is a strong solution to $(\ref{inner2})$ to $(\ref{initial data2})$ if 
\begin{flushleft}
(1) $z:[0,\tau)\to Z$ is continuous and $z(0)=(u_0,w_0)$ in $Z$, \\
(2) $z:(0,\tau)\to Z$ is Fr$\mathrm{\acute{e}}$chet differentiable, \\
(3) $t\in[0,\tau)\mapsto\Bigl(f\bigl(u(t),w(t)\bigr),g\bigl(u(t),w(t)\bigr)\Bigr)\in Z$ is well-defined, locally $\nu$-H$\mathrm{\ddot{o}}$lder continuous on $(0,\tau)$ and is continuous at $t=0$, \\
(4) for all $t\in (0,\tau)$, $u(t)\in W^{2,p}(\Omega)$, $u_e(t)\in W^{2,p}_{av}(\Omega)$, 
\end{flushleft}
and $(u,u_e,w)$ verify $(\ref{inner2})$-$(\ref{ODE2})$ for all $t\in(0,\tau)$ and for a.e. $x\in\Omega$, 
and the boundary conditions $(\ref{boundary2})$ and $(\ref{boundary3})$ for all $t\in(0,\tau)$ and for a.e. $x\in\partial\Omega$. 
}
\end{definition}

Let us consider bidomain equations as an abstract parabolic evolution equation on some Cartesian product spaces. 
We set 
\begin{align*}
&\mathcal{A}z:=(Au,0)~{\rm for}~z=(u,w)\in D(\mathcal{A}):=D(A)\times B^m, \\
&F:z\in Z\mapsto(f(z),g(z))\in Z, \\
&S:t\in[0,\infty)\mapsto(s(t),0)=\left(s_i(t)-A_i(A_i+A_e)^{-1}(s_i(t)+s_e(t)),0\right)\in Z, \\
&\mathcal{F}(t,z)=S(t)-F(z). 
\end{align*}
If one collects all calculation, then bidomain equations is transformed into 
\begin{align}
&\frac{dz}{dt}(t)+\mathcal{A}z(t)=\mathcal{F}(t,z(t)) \hspace{5mm} &{\rm in}&~Z\label{z}, \\
&u_e(t)=(A_i+A_e)^{-1}\left\{\left(s_i(t)+s_e(t)\right)-A_i P_{av}u(t)\right\} \hspace{5mm} &\in& \ D(A_e)\label{u_e}, \\
&z(0)=z_0\hspace{5mm} &{\rm in}&~Z. \nonumber 
\end{align}

\begin{lemma}[{\cite[Lemma 19]{BCP}}]
The function $z=(u,w)$ with $u_e$ is a strong solution $(\ref{inner2})$-$(\ref{initial data2})$ if and only if conditions {\rm (1)-(3)} of Definition $\ref{strong sol}$ and condition {\rm (4')} below is satisfied; 
\begin{flushleft}
{\rm (4')} for all $t\in(0,\tau)$, $u(t)\in D(A)$ satisfies $(\ref{z})$ and $(\ref{u_e})$. 
\end{flushleft}
\end{lemma}

We will use the general theory in Henry's book \cite{Hen}. 
We have to control the nonlinear term $f,g$. 
The key idea is to use fractional powers $\mathcal{A}^\alpha$ and related space $Z^\alpha$ with $0\le\alpha\le1$. 

\begin{definition}[\cite{Hen}]
If $\mathcal{A}$ is a sectorial operator in a Banach space $Z$ and if there is $a\ge0$ such that ${\rm Re}\,\sigma(\mathcal{A}+a)>0$, then for each $\alpha>0$ we define the operator 
\begin{align*}
(\mathcal{A}+a)^{-\alpha}:=\frac{1}{\Gamma(\alpha)}\int_0^\infty t^{\alpha-1}e^{-(\mathcal{A}+a)t} dt.  
\end{align*}
\end{definition}
For $\alpha>0$, we see $(\mathcal{A}+a)^{-\alpha}$ is a bounded linear operator on $Z$ which is one-to-one. 
By using this operator with fractional power, we define the domain $Z^\alpha$ of fractional power; 
\begin{align*}
&Z^\alpha:=R((\mathcal{A}+a)^{-\alpha})\hspace{5mm}(\alpha>0), \\
&\|x\|_{Z^{\alpha}}:=\|((\mathcal{A}+a)^{-\alpha})^{-1}x\|_Z. 
\end{align*}
For $\alpha=0$, we define $Z^0:=Z$, $\|x\|_{Z^0}:=\|x\|_Z$. 

\begin{remark}[\cite{Hen}]
\begin{itemize}
\item Different choices of $a$ give equivalent norms on $Z^\alpha$. 
\item $(Z^\alpha,\|\cdot\|_{Z^\alpha})$ is a Banach space, $Z^1=D(\mathcal{A})$ and for $0\le\beta\le\alpha\le1$, $Z^\alpha$ is a dense subspace of $Z^\beta$ with continuous inclusion. 
\end{itemize}
\end{remark}

\begin{lemma}[{\cite[Theorem 1.6.1]{Hen}}]\label{embed}
If $B=L^\infty(\Omega)$ and $f,g$ are locally Lipschitz continuous on $\mathbb{R}\times\mathbb{R}^m$, then 
\begin{align*}
Z^\alpha\subset L^\infty(\Omega)\times B^m \hspace{5mm} {\rm if}~\frac{d}{2p}<\alpha\le1, 
\end{align*}
and in that case, $F:z\in Z^\alpha\mapsto F(z)\in Z$ is locally Lipschitz continuous. \\
If $B=C^\nu(\Omega)$ and $f,g$ are $C^2$ functions on $\mathbb{R}\times\mathbb{R}^m$, then 
\begin{align*}
Z^\alpha\subset C^\nu(\Omega)\times B^m \hspace{5mm} {\rm if}~\frac{1}{2}\Bigl(\nu+\frac{d}{p}\Bigr)<\alpha\le1, 
\end{align*}
and in that case, $F:z\in Z^\alpha\mapsto F(z)\in Z$ is locally Lipschitz continuous. 
\end{lemma}

We are ready to state existence and uniqueness of the strong solution for bidomain equations. 
When $p=2$, $d=2,3$, this was proved in \cite[Theorem 20]{BCP} so our result is regarded as an extension of their result. 

\begin{theorem}[Local existence and uniqueness]
Let $0\le\alpha<1$ and  $1<p<\infty$ satisfying the relation in Lemma {\rm\ref{embed}}. 
Then for any $z_0=(u_0,w_0)\in Z^\alpha$, there exists $T>0$ such that bidomain equations have a unique strong solution on $[0,T)$. 
\end{theorem}

\begin{proof}
It is enough to show 
\begin{itemize}
\item $\mathcal{A}$ is a sectorial operator, 
\item $\mathcal{F}:[0,\infty)\times Z^\alpha\to Z$ is a locally H${\rm{\ddot{o}}}$lder continuous function in $t$ and a locally Lipschitz continuous function in $z$, 
\end{itemize}
because of existence and uniqueness theorem \cite[Theorem 3.3.3]{Hen}. 
First part is obvious since $\mathcal{A}=(A,0)$, $A$ is a sectorial operator and $0$ is a bounded linear operator. 
Note that a bounded linear operator is a sectorial operator and direct sum of a sectorial operator is a sectorial operator \cite{Hen}. 
Second part follows from the calculation as below. 
We need to show $s:[0,\infty)\to L^p(\Omega)$ is locally $\nu$-H${\rm{\ddot{o}}}$lder continuous in time. 
For any compact set $M\subset[0,\infty)$ there exists $C>0$ such that for all $t_1,t_2\in M$, we have 
\begin{align*}
\ &\|s(t_1)-s(t_2)\|_{L^p(\Omega)} \\
=&\|s_i(t_1)-s_i(t_2)-A_i(A_i+A_e)^{-1}(s_i(t_1)-s_i(t_2)+s_e(t_1)-s_e(t_2))\|_{L^p(\Omega)}\\
\le&\|s_i(t_1)-s_i(t_2)\|_{L^p(\Omega)}+C(\|s_i(t_1)-s_i(t_2)\|_{L^p(\Omega)}+\|s_e(t_1)-s_e(t_2)\|_{L^p(\Omega)}) \\
\le&C|t_1-t_2|^\nu. 
\end{align*}
Here, we invoked the fact that $A_i(A_i+A_e)^{-1}$ is a bounded linear operator and that $s_{i,e}$ are locally $\nu$-H${\rm{\ddot{o}}}$lder continuous functions. 
\end{proof}

We conclude this paper by studying regularity of a strong solution. 
Let $0<\nu<1$, $\Omega$ be a bounded $C^{2+\nu}$-domain, $f,g$ be $C^2$ regularity, and coefficient of $\sigma_{i,e}$ be $C^{1+\nu}(\overline{\Omega})$. 

\begin{theorem}[Regularity of a strong solution]
Consider the case $B=C^\nu(\Omega)$ in Definition {\rm \ref{strong sol}} and $0\le\alpha<1$ defined by Lemma {\rm\ref{embed}}. 
Assume that $s_{i,e}\in C_{loc}^\nu([0,\infty);L^p(\Omega))$ such that $s_{i,e}(t)\in C^\nu(\Omega)$ and $\int_\Omega (s_i(t)+s_e(t))dx=0(\forall t\ge0)$. 
For $z_0=(u_0,w_0)\in Z^\alpha$ the unique strong solution $z$ of bidomain equations defined on $[0,T)$ for some $T>0$ satisfies furthermore: 
\begin{flushleft}
{\rm (1)} For any $x\in\overline{\Omega}$, $u(x,\cdot)\in C^1((0,T);\mathbb{R})$ and $w(x,\cdot)\in C^1((0,T);\mathbb{R}^m)$. \\
{\rm (2)} For any $t\in(0,T)$, $u(\cdot,t),u_{i,e}(\cdot,t)\in C^2(\overline{\Omega})$. 
\end{flushleft}
\end{theorem}

\begin{proof}
We see that $t\in(0,T)\mapsto z(t)\in C^{\nu}(\Omega)\times(C^{\nu}(\Omega))^m$ is continuous (Fr$\mathrm{\acute{e}}$chet) differentiable. 
This actually implies that $(t,x)\in(0,T)\times\overline{\Omega}\mapsto z(x,t)=(u(x,t),w(t,x))$ is continuously differentiable in $t$. 
By {\cite[Theorem 3.5.2]{Hen}}, we have $t\in(0,T)\mapsto z(t)\in Z^{\nu}$ is continuously (Fr$\mathrm{\acute{e}}$chet) differentiable. 
This means $du/dt(t)\in C^{\nu}(\Omega)$. 
From (\ref{z}), 
\begin{align*}
P_{av}u(t)=A_e^{-1}(A_i+A_e)A_i^{-1}\left\{-\frac{du}{dt}(t)-f(u(t),w(t))+s(t)\right\}. 
\end{align*}
By elliptic regularity theorem for H${\rm{\ddot{o}}}$lder spaces, $P_{av}u(\cdot,t)$ is $(2+\nu)$-H${\rm{\ddot{o}}}$lder continuous since $-du/dt(t)-f(u(t),w(t))+s(t)$ is $\nu$-H${\rm{\ddot{o}}}$lder continuous. 
Therefore $u(\cdot,t)$ is in $C^2(\overline{\Omega})$.
The function $u_e$ is also in $C^2(\overline{\Omega})$ by (\ref{u_e}). 
\end{proof}

\end{document}